\documentclass[12pt]{article}
\usepackage[english,activeacute]{babel}
\usepackage{amsmath,amsfonts,amssymb,amstext,amsthm,amscd,mathrsfs,amsbsy,graphics,graphicx}
\usepackage{xypic}
\usepackage[all]{xy}
\usepackage{url}
\newtheorem{theo}{Theorem}[section]
\newtheorem{pro}[theo]{Proposition}
\newtheorem{coro}[theo]{Corollary}
\newtheorem{lem}[theo]{Lemma}

\theoremstyle{definition}
\newtheorem{defi}[theo]{Definition}
\newtheorem{exam}[theo]{Example}
\newtheorem{rem}[theo]{Remark}

\begin{document}

\title{On Deligne's functorial Riemann-Roch theorem in positive characteristic}
\author{\\Quan XU\\Universit\'{e} Paul Sabatier}\maketitle
\footnote{Email to: xuquan.math@gmail.com}

%\pagenumbering{roman}

%\newpage \tableofcontents \newpage

%\frontmatter

\begin{abstract}
In this note, we give a proof for a variant of the functorial Deligne-Riemann-Roch theorem in positive characteristic based on 
ideas appearing in Pink and R\"ossler's proof of the Adams-Riemann-Roch theorem in positive characteristic (see \cite{Pi}).
The method of their proof appearing in \cite{Pi}, which is valid for any positive characteristic and which is completely different from the 
classical proof, will allow us to prove the functorial Deligne-Riemann-Roch theorem in a much easier and more direct way. 
Our proof is also partially compatible with Mumford's isomorphism. 
\end{abstract}

\section{Introduction}

In \cite{Pi}, Richard Pink and Damian R\"{o}ssler proved the following version of the Adams-Riemann-Roch theorem:

Let $f:X\rightarrow Y$ be projective and smooth of relative dimension $r$, where
$Y$ is a quasi-compact scheme of characteristic $p>0$ and carries an ample invertible sheaf. Then the following equality
$$\psi^{p}(R^{\bullet}f_{*}(E))=R^{\bullet}f_{*}(\theta ^{p}(\Omega_{f})^{-1}\otimes\psi^{p}(E))~~~~~(*)$$
holds in $K_{0}(Y)[\frac{1}{p}]:=K_{0}(Y)\otimes_{\mathbb{Z}}\mathbb{Z}[\frac{1}{p}].$

The symbols in the previous equality are explained simply as follows:

(a) The symbol $\psi^{p}$ is the $p$-th Adams operation and $\theta ^{p}$ is the $p$-th Bott class operation (see Sect. 3.2 and 3.3);

(b) For a vector bundle $E$, $R^{\bullet}f_{*}(E)=\sum_{i\geq 0}(-1)^{i}R^{i}f_{*}(E)$ where $R^{i}f_{*}$ is the higher direct image functor of the push-forward $f_{*}$ (see Sect. 2.1);

(c) For a quasi-compact scheme $Y$, $K_{0}(Y)$ is the Grothendieck group of locally free coherent sheaves of $\mathcal{O}_{Y}$-modules (see Sect. 2.1);

(d) The symbol $\Omega_{f}$ is the relative differentials of the morphism $f$. When $f$ is a smooth and projective morphism, $\Omega_{f}$ is locally free sheaf (see \cite{Hart}).

For the general case, i.e., for a projective local complete intersection morphism $f$ (see \cite{FL}, Pag. 86) with no restrictions on the characteristic, $p$ can be replaced by any positive integer $k\geq2$ in the equality (*) 
and $\Omega_{f}$ is replaced by cotangent complex of the morphism $f$ (see \cite{Lit}). Then the previous equality  holds in 
$K_{0}(Y)\otimes_{\mathbb{Z}}\mathbb{Z}[\frac{1}{k}]$. The classical proof of the Adams-Riemann-Roch theorem for a projective local complete intersection morphism
consists in verifying that the theorem holds for both closed immersions and projections and also for their composition. Moreover, 
a deformation to a normal cone is used (see \cite{Man} or \cite{FL}). However, in the case of characteristic $p$, the decomposition 
of the projective morphism and the deformation are completely avoidable in the course of the proof of the $p$-th Adams Riemann Roch theorem, i.e., the equality $(*)$, when $f$ is a projective and smooth morphism. In this case, Pink and R\"ossler construct an explicit representative 
for the $p$-th Bott element $\theta ^{p}(E)=\tau(E):=\text{Sym}(E)/\mathcal{J}_{E}$ in $K_{0}(Z)$ for any locally free coherent sheaf $E$ over a quasi-compact scheme $Z$ of characteristic
$p$ where $\text{Sym}(E)$ is the symmetric algebra of $E$ and $\mathcal{J}_{E}$ is the graded sheaf of ideals of $\text{Sym}(E)$ that is locally generated
by the sections $e^{p}$ of $\text{Sym}(E)$ for all the sections $e$ of $E$ (see Sect. 3.2). 
Furthermore, they proved an isomorphism of $\mathcal{O}_{X}$-modules
$$I/ I^{2}\cong \Omega_{f},$$
and an isomorphism of graded $\mathcal{O}_{X}$-algebras:
$$\tau(I/ I^{2})\cong Gr(F^{*}F_{*}\mathcal{O}_{X}).$$ 
(The previous notations will be explained in Section \ref{section ARR char p}). These isomorphisms play an essential role in their proof
of the Adams-Riemann-Roch theorem in positive characteristic, which is also very important in our theorem.

When $f:C\rightarrow S$ is a smooth family of proper curves, Deligne proved the following functorial Riemann Roch theorem (see \cite{Del}, Theorem 9.9):
There exists a unique, up to sign, functorial isomorphism of line bundles
\begin{align}
(\det &Rf_{*}L)^{\otimes 18 }\notag\\
&\cong (\det  Rf_{*}\mathcal{O})^{\otimes 18 }\otimes (\det  Rf_{*}(L^{\otimes 2}\otimes 
\omega^{-1}))^{\otimes 6}\otimes  (\det  Rf_{*}(L\otimes \omega^{-1}))^{\otimes (-6)}.\notag
\end{align}
The statement above is not Deligne's original statement, but the essence is the same. We will explain it in Thm. \ref{Del,func} and Rem. \ref{explai}.

In this note, our strategy is to give a similar isomorphism in positive characteristic for a line bundle. More precisely, we will provide an isomorphism between $(\det Rf_{*}L)^{\otimes p^{4}}$ and some tensor products of $(\det Rf_{*}(L^{\otimes l}\otimes \omega ^{\otimes n}))^{\otimes m}$ for any line bundle $L$ and some integers $l, n, m$, in any characteristic $p>0$, by using only 
a property of the Deligne pairing in \cite{Del} as well as some ideas from the proof of the Adams-Riemann-Roch theorem in positive characteristic 
given in \cite{Pi}. More importantly, the isomorphism is also stable under base change. These are our main results (see Theorem \ref{functor}). 
A byproduct of our result is a partial compatibility with Mumford's isomorphism. We will give a brief introduction to  Mumford's isomorphism and verify the 
compatibility (see Cor. \ref{mfi}).

Furthermore, we also make a comparison between Deligne's functorial Riemann Roch theorem and our result. 
In the case of characteristic $p=2$, our result completely coincides with Deligne's theorem. When the characteristic is an odd prime number, our theorem is just an  analogue of Deligne's.
The possible reason is that our setting emphasizes a lot about the case of characteristic $p$. However, Deligne's work in \cite{Del} is independent of the characteristic.

\section{Preliminaries}
We always assume that $X$ is a quasi-compact scheme whenever a scheme is mentioned in this note, unless we use a different statement for the scheme.
\subsection{Grothendieck groups and the virtual category}
Let $X$ be a scheme. We denote by $\mathcal{I}$ the category of coherent sheaves on $X$ and
 by $\mathcal{L}$  its fully subcategory of locally free
sheaves . Furthermore, let $\mathbb{Z}[\mathcal{I}]$ (respectively $\mathbb{Z}[\mathcal{L}]$) denote the free abelian group generated by the isomorphism classes $[\mathcal{F}]$ of $\mathcal{F}$ in category
$\mathcal{I}$(respectively $\mathcal{L}$).

For the exact sequence $0\rightarrow \mathcal{F}_{1}\rightarrow\mathcal{F}_{2}\rightarrow\mathcal{F}_{3}\rightarrow 0$
of sheaves of $\mathcal{I}$ (respectively $\mathcal{L}$), we form the element $\mathcal{F}_{2}-\mathcal{F}_{1}-\mathcal{F}_{3}$ and consider the subgroup $\mathcal{J}$ (respectively $\mathcal{J}_{1}$) generated by such elements 
in $\mathbb{Z}[\mathcal{I}]$ (resp. $\mathbb{Z}[\mathcal{L}]$).

\begin{defi}\label{b. 1}
  We define $K_{0}(X):=\mathbb{Z}[\mathcal{L}]/\mathcal{J}_{1}$ and $K_{0}^{'}(X):=\mathbb{Z}[\mathcal{I}]/\mathcal{J}$. \\ 
Usually, $K_{0}(X)$ and $K_{0}^{'}(X)$ are called the Grothendieck groups of locally free shaves and coherent sheaves on a scheme $X$, respectively.
\end{defi}
The following are basic facts about Grothendieck groups:

(1) The tensor product of $\mathcal{O}_{X}$-modules makes the group $K_{0}(X)$ into a commutative unitary ring and the inverse image of locally free sheaves under
any morphism of schemes $X^{'}\rightarrow X$ induces a morphism of unitary rings $K_{0}(X)\rightarrow K_{0}(X^{'})$ (see \cite{Man}, \S 1);

(2) The obvious group morphism $K_{0}(X)\rightarrow K_{0}^{'}(X)$ is an isomorphism if $X$ is regular and carries an ample invertible sheaf (see \cite{Man}, Th. I.9);

(3) Let $f:X\rightarrow Y$ be a projective local complete intersection morphism of schemes (A morphism $f 
: X\rightarrow Y$ is called a  local complete intersection morphism if $f $  is a composition of morphisms as $X\rightarrow \text{P} \rightarrow Y$ where the first morphism is a regular embedding and the second is a smooth morphism. See \cite{Ful}  or \cite{FL}) and $Y$ carries an ample invertible sheaf. There is
a unique group morphism $\text{R}^{\bullet}f_{*}:K_{0}(X)\rightarrow K_{0}(Y)$ which sends the class of a locally free coherent sheaf $E$ on $X$ to the class of the class of the strictly perfect complex (The strictly perfect complex will be defined in Sect. 2.2) $\text{R}^{\bullet}f_{*}E$ in $K_{0}(Y)$, where $\text{R}^{\bullet}f_{*}E$ is defined 
to be $\sum_{i\geq 0}(-1)^{i}\text{R}^{i}f_{*}E$ and $\text{R}^{i}f_{*}E$ is viewed as an element in $K_{0}(Y)$ (see \cite{Berth}, IV, 2.12).

In \cite{Del}, Deligne defined a categorical refinement of the Grothendieck groups. In order to define it, we
need to review some material from the theory of exact categories. Above all, let us recall definitions of the additive category and the abelian category.
\begin{defi}
An additive category is a category $\mathcal{A}$ in which $\text{Hom}(A,B)$ is an abelian group for all objects $A, B$, composition of arrows is bilinear, and  
$\mathcal{A}$ has (finite) direct sums and a zero object. An abelian category is an additive category in which every arrow $f$ has a kernel, co-kernel, image and co-image, and the canonical
map $\text{coim}(f)\rightarrow \text{im}(f)$ is an isomorphism.
\end{defi}
Let $\mathcal{A}$ be an additive category. A short sequence in $\mathcal{A}$ is a pair of
composable morphisms $L\rightarrow M\rightarrow N$ such that $L\rightarrow M $ is a kernel for $M\rightarrow N$
and $M\rightarrow N$ is a cokernel for $L \rightarrow M$. Homomorphisms of short sequences are defined in the obvious way as commutative diagrams.

\begin{defi}
 An exact category is an additive category $\mathcal{A}$ together with a choice $S$ of
a class of short sequences, called short exact sequences, closed under isomorphisms and satisfying the axioms below. A short exact sequence is displayed as
$L \rightarrowtail M \twoheadrightarrow N$, where $L \rightarrowtail M$ is called an admissible monomorphism and $M\twoheadrightarrow N$ is
called an admissible epimorphism. The axioms are the following:

(1) The identity morphism of the zero object is an admissible monomorphism
and an admissible epimorphism.

(2) The class of admissible monomorphisms is closed under composition and
cobase changes by push-out along arbitrary morphisms, i.e., given any admissible monomorphism $L\rightarrowtail M$ and any arbitrary $L\rightarrow L^{'}$ , their push-out $M^{'}$ exists and the induced morphism $L^{'} \rightarrow M^{'}$ is again an admissible
monomorphism.
$$\xymatrix{L\ar @{>->}[r]\ar[d]&M\ar@ {-->}[d]\\
L^{'}\ar@{>-->}[r]& M^{'}}$$

(3) Dually, the class of admissible epimorphisms is closed under composition
and base changes by pull-backs along arbitrary morphisms, i.e., given any
admissible epimorphism $M\twoheadrightarrow
N$ and any arbitrary $N^{'}\rightarrow N$ , their pullback
$M^{'}$ exists and the induced morphism $M^{'}\rightarrow N^{'}$ is again an admissible
epimorphism.
$$\xymatrix{M^{'}\ar @{-->}[r]\ar @{-->>}[d]&M\ar @{->>}[d]\\
N^{'}\ar[r]& N
}$$
\end{defi}
\begin{exam}
Any abelian category is an exact category in an evident way. Any additive  category can be made into an exact category in at least one way by taking $S$ be the family
of split exact sequences. 
\end{exam}
 
In order to give the definition of the virtual category, we shall need the definition of a groupoid, especially of a specific groupoid called the Picard groupoid (see \cite{Del}, \S 4).
\begin{defi}
A groupoid is a (small) category in which all morphisms are invertible.
\end{defi}
This means there is a set $B$ of objects, usually called the base, and a set $G$ of morphisms, usually called the arrows. One
says that $G$ is a groupoid over $B$ and writes $\xymatrix{G\ar@<+.7ex> [r]\ar@<-.7ex>[r]& B}$ or just $G$ when the base is understood.
We can be much more explicit about the structure of a groupoid. To begin with, each arrow has an associated source object and associated
target object. This means that there are two maps $$s, t: G\rightarrow B$$ called the source and the target, respectively. Since a groupoid
is a category, there is a multiplication of arrows $$m: G\times_{B} G\rightarrow G$$ where  $G\times_{B} G$ fits into the pull-back
square:  $$\xymatrix{G\times_{B} G\ar[d]\ar[r]& G\ar[d]^{s}\\
                     G\ar[r]^{t}&B      }$$

More explicitly,  $$G\times_{B} G=\{(h,g)\in G\times G\lvert s(h)=t(g)\}=(s\times t)^{-1}(\Delta_{B}).$$
This is just to say that we can only compose arrows when the target of the first and source of the second agree. This multiplication preserves
sources and targets: $$ s(hg)=s(g), t(hg)=t(h),$$ and is associative:
                                       $$k(hg)=(kh)g.$$

For each object $x\in B$, there is an identity arrow, written $1_{x}\in G$ and this association defines an injection $$ \mathbf{1}: B\hookrightarrow G.$$\\
For each arrow $g\in G$, there is an inverse arrow, written $g^{-1}\in G$, and this defines a bijection $$\iota : G\rightarrow G.$$\\
These identities and inverse satisfy the usual properties. Namely, identities work as expected: $$1_{t(g)}g=g=g1_{s(g)},$$
and  inversion swaps sources and targets: $$s(g^{-1})=t(g),~t(g^{-1})=s(g),$$
and inverses works as expected, with respect to the identities: $$g^{-1}g=1_{s(g)},~ gg^{-1}=1_{t(g)}.$$
Thus we have a set of maps between $B$ and $G$ as follows:$$\xymatrix{ B\ar @{.>}[r]&G\circlearrowleft\iota\ar@<+.9ex> [l]^{s}\ar@<-.9ex>[l]_{t}}$$

\begin{exam}Any set $X$ can be viewed as a groupoid over itself, where the only arrows are identities. This is the trivial groupoid, or the unit groupoid
and is simply written as $X$. The source and target maps are the identity map $\text{id}_{X}$, multiplication is only defined between a point and itself: $$xx=x.$$
\end{exam}

 \begin{exam} Any set gives rise to the pair groupoid of $X$. The base is $X$, and the set of arrows is $X\times X \rightrightarrows X.$ The source and target
maps are the first and second projection maps. Multiplication is defined as follows: $(x,x^{'})(x^{'},x{''})=(x,x{''}).$
\end{exam}

\begin{defi}

A Picard category is a groupoid $P$ together with the following extra structure:\\
1. A functor $+ : P\times P\rightarrow P.$\\
2. An isomorphism of functors:  $
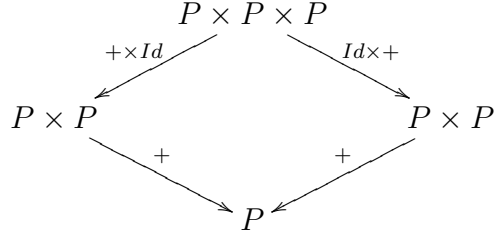
\xymatrix{&P\times P \times P\ar[dl]_{+\times Id}\ar[dr]^{Id\times +}&\\
                                   P\times P\ar[dr]^{+}&  &P\times P\ar[dl]_{+}\\
                                    & P  &    }$

$ \sigma_{x,y,z}: (x+y)+z\backsimeq x+(y+z).$\\
3. A natural transformation $\tau_{x,y}: x+y\backsimeq y+x$ commuting with $+$.\\
4. For all $x\in P$, the functor $P\rightarrow P$ by $y\mapsto x+y$ is an equivalence.\\
5. Pentagon Axiom: The following diagram commutes\\
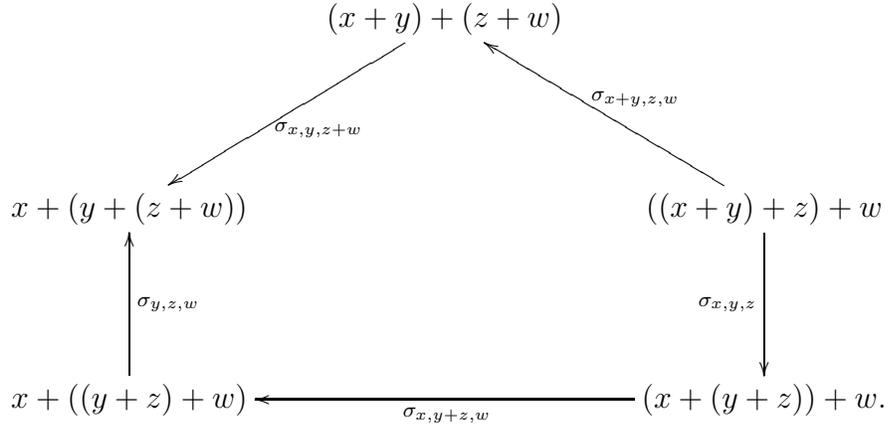
$$\xymatrix{  & (x+y)+(z+w)\ar[ddl]^{\sigma_{x,y,z+w}}& \\
                  & & \\
                    x+(y+(z+w)) & &  ((x+y)+z)+w\ar[uul]_{\sigma_{x+y,z,w}}\ar[dd]_{\sigma_{x,y,z}}\\
                      & &\\
        x+((y+z)+w)\ar[uu]_{\sigma_{y,z,w}}& &(x+(y+z))+w.\ar[ll]^{\sigma_{x,y+z,w}}                   }$$
6. $\tau_{x,x}$=id for all $x\in P$\\
7. $\forall x,y\in P, \tau_{x,y}\tau_{y,x}$=id\\
8. Hexagon Axion: The following diagram commutes:\\
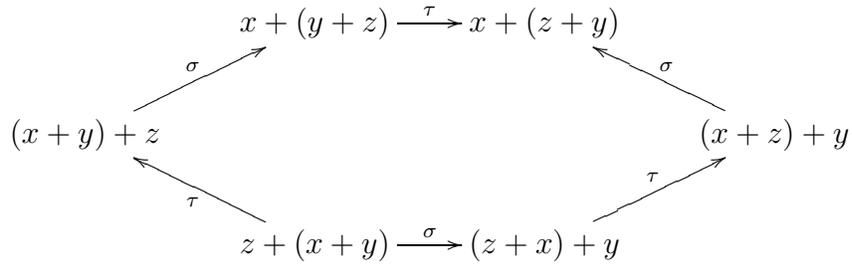
$$\xymatrix{    & x+(y+z)\ar[r]^{\tau}& x+(z+y)\\
      (x+y)+z\ar[ur]^{\sigma}& & & (x+z)+y\ar[ul]_{\sigma}\\
    & z+(x+y)\ar[ul]^{\tau}\ar[r]^{\sigma}& (z+x)+y\ar[ur]^{\tau}&                    }.$$

\end{defi}

\begin{exam}\label{lbd}
Let $X$ be a scheme. We denote by $\mathscr{P}_{X}$ the category of graded invertible 
$\mathcal{O}_{X}$-modules. An object of $\mathscr{P}_{X}$ is a pair $(L,\alpha)$ where $L$ is an invertible $\mathcal{O}_{X}$-module and $\alpha$
is a continuous function: $$\alpha: X\rightarrow \mathbb{Z}.$$ 

A homomorphism $h:(L,\alpha)\rightarrow (M,\beta)$ is a homomorphism of $\mathcal{O}_{X}$-modules such that for each $x\in X$ we have:
$$\alpha(x)\neq\beta(x)\Rightarrow h_{x}=0.$$

We denote by $\mathscr{P}is_{X}$ the subcategory of $\mathscr{P}_{X}$ whose morphisms are all isomorphism.
The tensor product of two objects in $\mathscr{P}_{X}$ is given by: $$(L,\alpha)\otimes(M,\beta)=(L\otimes M, \alpha+\beta).$$
For each pair of objects $(L,\alpha), (M,\beta)$ in $\mathscr{P}_{X}$  we have an isomorphism:
$$\xymatrix{\psi_{(L,\alpha), (M,\beta)}: (L,\alpha)\otimes(M,\beta)\ar[r]^(.60){\sim}&(M,\beta)\otimes(L,\alpha)  }$$
defined as follows: If $l\in L_{x}$ and $m\in M_{x}$ then 
$$\psi(l\otimes m)=(-1)^{\alpha(x)+\beta(y)}\cdot m\otimes l.$$\\
Clearly: $$\psi_{(M,\beta),(L,\alpha)}\cdot\psi_{(L,\alpha), (M,\beta)}=1_{(L,\alpha)\otimes(M,\beta)}$$\\
We denote by $1$ the object $(\mathcal{O}_{X},0)$. A right inverse of an object $(L,\alpha)$ in $\mathscr{P}_{X}$ will be an object $(L^{'},\alpha^{'})$
together with an isomorphism $$\xymatrix{\delta: (L,\alpha)\otimes(L^{'},\alpha^{'})\ar[r]^(.70){\sim}& 1}$$\\
Of course $\alpha^{'}=-\alpha$.
A right inverse will be considered as a left inverse via:
$$\xymatrix{\delta: (L^{'},\alpha^{'})\otimes(L,\alpha)\ar[r]_(.55){\sim}^(.55){\psi}& (L,\alpha)\otimes(L^{'},\alpha^{'})\ar[r]_(.70){\sim}^(.70){\delta}&1 }.$$
Given the definition, further verification implies that $\mathscr{P}is_{X}$ is a Picard category.
\end{exam}
After defining the exact category, we can give the definition of Deligne's virtual category. 
By an admissible filtration in an exact category we mean a finite sequence of admissible monomorphisms $0=A^{0}\rightarrowtail A^{1}\rightarrowtail
\cdots \rightarrowtail A^{n} = C.$
\begin{defi}\label{vircon}(see \cite{Del}, Pag. 115)
The virtual category $V(\mathcal{C})$ of an exact category $\mathcal{C}$ is
a Picard category, together with a functor $\{~\} : (\mathcal{C}, iso)\rightarrow V(\mathcal{C})$ (Here, the
first category is the subcategory of $\mathcal{C}$ consisting of the same objects and the
morphisms are the isomorphisms of $\mathcal{C}$.), with the following universal property:\\
Suppose we have a functor $[~] : (\mathcal{C}, iso)\rightarrow P$ where $P$ is a Picard category,
satisfying

(a) Additivity on exact sequences, i.e., for an exact sequence $A \rightarrow B \rightarrow C $ ($A\rightarrow B$ is a admissible monomorphism and $B\rightarrow C$ is a admissible epimorphism), 
we have an isomorphism $[B]\cong[A]+[C]$, functorial with respect to isomorphisms of exact sequences.

(b) A zero-object of $C$ is isomorphically mapped to a zero-object in $P$ (According to (4) in the definition of  the Picard category, it implies the existence of the 
unit object which is also called zero-object. See \cite{Del}, \S 4.1.).

(c) The additivity on exact sequences is compatible with admissible filtrations, i.e., for an admissible filtration $C\supset B\supset A\supset 0$,
the diagram of isomorphisms from (a) 
$$\xymatrix{[C]\ar[rr]\ar[d]&    & [A]+ [C/A]\ar[d]\\
    [B]+[C/B]\ar[rr] & & [A]+[B/A]+[C/B]                  }$$
 is commutative.

(d) If $f: A\rightarrow B$ is an isomorphism and $\sum$ is the exact sequence $0\rightarrow A\rightarrow B$ (resp. $A\rightarrow B\rightarrow 0$ ), then $[f]$ (resp. $[f]^{-1}$) 
is the composition
$$\xymatrix{[A]\ar[r]_(.35){\sum}&[0]+[B]\ar[r]_(.6){(b)}&[B]   }$$
$$(resp. \xymatrix{[B]\ar[r]_(.35){\sum}&[A]+[0]\ar[r]_(.6){(b)}&[A]})$$
where (b) in the diagram above means that the morphism is from (b).\\
Then the conclusion is that the functor $[~] :(\mathcal{C}, iso) \rightarrow P$ factors uniquely up to
unique isomorphism through $(\mathcal{C}, iso)\rightarrow V(\mathcal{C})$. 
\end{defi}
Roughly speaking, for an exact category $\mathcal{C}$, $V(\mathcal{C})$ is a universal Picard category with a functor $[~]$ satisfying some properties. In practice, the functor $[~]$ usually can be chosen as the determinant functor we will 
define in the next subsection.
\begin{rem}\label{v,k}
 In \cite{Del}, Deligne also provided a topological definition for the virtual category  of a small exact category.
The category of virtual objects of $\mathcal{C}$, $V(\mathcal{C})$ is the following: Objects
are loops in $\mathcal{B}Q\mathcal{C}$ around a fixed zero-point, and morphisms are homotopy-classes of homotopies
of loops. Here $\mathcal{B}Q\mathcal{C}$, is the geometrical realization of the Quillen Q-construction of $\mathcal{C}$.
The addition is the usual addition of loops. This construction is the fundamental groupoid of the loop
space $\Omega\mathcal{B}Q\mathcal{C}$ of $\mathcal{B}Q\mathcal{C}$. By the description above and Quillen's definition of $K$-theory (see \cite{Quill}), the group of isomorphism classes of objects
of the virtual category is the usual Grothendieck group $K_{0}(X)$ of the category of vector bundles on $X$ (see \cite{Del}, Pag. 114).
\end{rem}

\subsection{The determinant functor}\label{sdf}

In this subsection, we will consider the determinant functor and mainly consult \cite{Kund}. In \cite{Kund}, the determinant functor can be defined in several backgrounds.
But the case we are most interested in  is the determinant functor from some subcategory of derived category to the subcategory of the category $\mathscr{P}_{X}$ of graded line bundles. 

In the following,  we denote by $\mathscr{C}_{X}$ the category  of finite locally free $\mathcal{O}_{X}$-modules for a scheme $X$.
 
\begin{defi}
If $E\in \text{ob}(\mathscr{C}_{X})$, we define: $\det ^{*}(F)=(\wedge^{max}F, \text{rank}F)$ \\ (where $(\wedge^{max}F)_{x}=\wedge^{\text{rank}F_{x}}F_{x}$).
\end{defi}

For every short exact sequence of objects in $\mathscr{C}_{X}$ 
$$\xymatrix{0\ar[r]&F^{'}\ar[r]^{\alpha}& F\ar[r]^{\beta}&F^{''}\ar[r]&0                 }$$ we have an isomorphism,
$$\xymatrix{i^{*}(\alpha,\beta): \det^{*}F^{'}\otimes \det^{*}F^{''}\ar[r]^(.70){\sim}&\det^{*} F   },$$ such that locally,
$$i^{*}(\alpha,\beta)((e_{1}\wedge\ldots\wedge e_{l})\otimes(\beta f_{1}\wedge\ldots\beta f_{s}))=\alpha e_{1}\wedge\ldots\alpha e_{l}\wedge f_{1}\wedge\ldots f_{s}$$
for $e_{i}\in\Gamma (U,F^{'})$ and $f_{j}\in\Gamma (U,F^{''})$.

\begin{defi}
 If $F^{i}$ is an indexed object of $\mathscr{C}_{X}$ we define:
\[
\det(F^{i})=
\begin{cases}
\det^{*}(F^{i}) &\text{if $i$ even};\\
 \det^{*}(F^{i})^{-1}    &\text{if $i$ odd}.

\end{cases}
\]
If $$\xymatrix{ 0\ar[r]& F^{i^{'}}\ar[r]^{\alpha^{i}}&F^{i}\ar[r]^{\beta^{i}}&F^{i^{''}}\ar[r]&0} $$
is an indexed short exact sequence of objects in $\mathscr{C}_{X}$, we define
\[
i(\alpha^{i},\beta^{i})=
\begin{cases}
i^{*}(\alpha^{i},\beta^{i})&\text{if $i$ even};\\
 i^{*}(\alpha^{i},\beta^{i})^{-1}   &\text{if $i$ odd}.

\end{cases}
\]
Usually, for a object $F$ in $\mathscr{C}_{X}$, we view the object as the indexed object by $0$, i.e., $\det(F)=\det^{*}(F)$.

\end{defi}
We also denote by $\mathscr{C}^{\cdot}_{X}$ the category of the bounded complex of objects in $\mathscr{C}_{X}$ over a scheme $X$.
\begin{defi}
If $F^{\cdot}$ is an object of $\mathscr{C}^{\cdot}_{X}$, we define $$\det(F^{\cdot})=\cdots\otimes\det(F^{i+1})\otimes\det(F^{i})\otimes\det(F^{i-1})\otimes\cdots$$
Furthermore, if $$\xymatrix{ 0\ar[r]& F^{\cdot'}\ar[r]^{\alpha}&F^{\cdot}\ar[r]^{\beta}&F^{\cdot''}\ar[r]&0} $$
is a short exact sequence of objects in $\mathscr{C}^{\cdot}_{X}$ we define 
$$\xymatrix{i(\alpha,\beta):\det(F^{\cdot'})\otimes \det(F^{\cdot''})\ar[r]^(.70){\sim}& \det(F^{\cdot})  }$$
to be the composite:
$$\det(F^{\cdot'})\otimes \det(F^{\cdot''})=\cdots\otimes\det(F^{i'})\otimes\det(F^{i-1'})\otimes\cdots$$
$$\xymatrix{\otimes\det(F^{i''})\otimes\det(F^{i-1''})\otimes\cdots\ar[r]^(.55){\sim}&\cdots\otimes\det(F^{i'})\otimes\det(F^{i''}) }$$
$$\xymatrix @C=0.7in{\otimes\det(F^{i-1'})\otimes\det(F^{i-1''})\otimes\cdots\ar[r]^(.65){\otimes_{i}i(\alpha^{i},\beta^{i})}_(.65){\sim}&\cdots\otimes\det(F^{i})  }$$
$$\otimes\det(F^{i-1})\otimes\cdots=\det(F^{\cdot}) .$$
\end{defi}
In \cite{Kund}, it is proved that there is one and, up to canonical isomorphism, only one determinant $(f,i)$ from $\mathscr{C}is_{X}$ (resp. $\mathscr{C}^{\cdot}is_{X}$) to $\mathscr{P}is_{X}$, which we write $(\det, i)$, 
where $\mathscr{C}is_{X}$ (resp. $\mathscr{C}^{\cdot}is_{X}$) is the category with same objects from $\mathscr{C}_{X}$ (resp. $\mathscr{C}^{\cdot}_{X}$) and the morphisms being
all isomorphisms (resp. quasi-isomorphisms). In case of repeating, we don't give the definitions of the determinant functor from from $\mathscr{C}is_{X}$ (resp. $\mathscr{C}^{\cdot}is_{X}$) to $\mathscr{P}is_{X}$, because the definitions are completely similar
to the following definition of the extended functor. For the precise definitions and proofs, see \cite{Kund}, Pag. 21-30.

In order to extend the determinant functor to the derived category in \cite{Kund}, we need to recall the definitions about the perfect complex and the strictly perfect complex:

In [11], a perfect complex $\mathcal{F}^{\cdot}$ on a scheme $X$ means a complex of $\mathcal{O}_{X}-$modules (not necessarily quasi-coherent) such
that locally on $X$ there exists a bounded complex $\mathcal{G}^{\cdot}$ of finite free $\mathcal{O}_{X}-$modules and a quasi-isomorphism:
$$\mathcal{G}^{\cdot}\rightarrow\mathcal{F}^{\cdot}\mid_{U}$$ for any open subset $U$ of a covering of $X$. A strictly perfect complex $\mathcal{F}^{\cdot}$ on a scheme $X$
is a bounded complex of locally free $\mathcal{O}_{X}-$modules of finite type. In other words, a perfect complex is locally quasi-isomorphic to a strictly perfect complex.
We denote by $\text{Parf}_{X}$ the full subcategory of $\text{D(Mod}X)$ whose objects are perfect complexes and denote by $\text{Parf-is}_{X}$ the subcategory of
$\text{D(Mod}X)$  whose objects are perfect complexes and morphisms are only quasi-isomorphisms.
\begin{defi}\label{ex,de}(see \cite{Kund}, Pag. 40)
An extended determinant functor $(f, i)$ from  $\text{Parf-is}$ to $\mathscr{P}is$ consist of the following data:

I) For every scheme $X$, a functor
$$f_{X}: \text{Parf-is}_{X}\rightarrow\mathscr{P}is_{X}$$ such that $f_{X}(0)=1$.

II) For every short exact sequence of complexes $$ \xymatrix{ 0\ar[r]& F\ar[r]^{\alpha}&G\ar[r]^{\beta}&H\ar[r]&0} $$ in $\text{Parf-is}_{X}$,
 we have an isomorphism:
$$i_{X}(\alpha,\beta):\xymatrix{f_{X}(F)\otimes f_{X}(H)\ar[r]^(.65){\sim}&f_{X}(G)}$$ such that for the particular short exact sequences
$$\xymatrix {0\ar[r]&H\ar@{=}[r]&H\ar[r]&0\ar[r]&0}$$
and 
$$\xymatrix {0\ar[r]&0\ar[r]&H\ar@{=}[r]&H\ar[r]&0}$$
we have : $i_{X}(1,0)=i_{X}(0,1)=1_{f_{X}(H)}$.\\
We require that:

i) Given an isomorphism of short exact sequences of complexes
$$\xymatrix{0\ar[r]&F\ar[r]^{\alpha}\ar[d]^{u}&G\ar[r]^{\beta}\ar[d]^{v}&H\ar[r]\ar[d]^{w}&0 \\
0\ar[r]&F^{'}\ar[r]^{\alpha ^{'}}&G^{'}\ar[r]^{\beta ^{'}}&H^{'}\ar[r]&0  }$$
the diagram 
$$\xymatrix{f_{X}(F)\otimes f_{X}(H)\ar[r]^(.65){i_{X}(\alpha,\beta)}_(.65){\sim}\ar[d]^{f(u)\otimes f_{X}(w)}_{\wr}&f(G)\ar[d]^{f_{X}(v)}_{\wr}\\
f(F^{'})\otimes f(H^{'})\ar[r]_(.65){i_{X}(\alpha^{'},\beta^{'})}^(.65){\sim} &  f(G^{'})      }$$
commutes.

ii) Given a exact sequence of short exact sequences of complexes, i.e., a commutative diagram
$$\xymatrix{    &   0\ar[d]&0\ar[d]&0  \ar[d]\\
0\ar[r]&F\ar[r]^{\alpha}\ar[d]^{u}&G \ar[r]^{\beta}\ar[d]^{u^{'}}&H\ar[r]\ar[d]^{u^{''}}&0 \\
0\ar[r]&F^{'}\ar[r]^{\alpha^{'}}\ar[d]^{v}&G^{'}\ar[r]^{\beta^{'}}\ar[d]^{v^{'}}&H^{'}\ar[r]\ar[d]^{v ^{''}}&0 \\
0\ar[r]&F^{''}\ar[r]^{\alpha^{''}}\ar[d]&G^{''} \ar[r]^{\beta^{''}}\ar[d]&H^{''}\ar[r]\ar[d]&0 \\
            &   0  &0  &0       }$$
the diagram:
$$\xymatrix{f_{X}(F)\otimes f_{X}(H)\otimes f_{X}(F^{\cdot''})\otimes f_{X}(H^{''})\ar[rrr]^(.62){i_{X}(\alpha,\beta)\otimes i_{X}(\alpha^{''},\beta^{''})}_(.62){\sim}\ar[d]^{i_{X}(u,v)\otimes i_{X}(u^{''},v^{''})}_{\wr}&  & &f_{X}(F^{\cdot})\otimes f_{X}(H^{\cdot})\ar[d]^{i_{X}(u^{'}, v^{'})}_{\wr}\\
f_{X}(F^{'})\otimes f_{X}(H^{'})\ar[rrr]_(.65){i_{X}(\alpha^{'},\beta^{'})}^(.65){\sim} & &  &f_{X}(G^{'})      }$$
commutes.

iii) $f$ and $i$ commutes with base change. More precisely, this means: \\
For every morphism of schemes 
$$g: X\rightarrow Y$$
we have an isomorphism 
$$\eta(g):  \xymatrix{f_{X}\cdot \text{L}g^{*}\ar[r]^{\sim}&g^{*}f_{X}}$$ such that for every short exact sequence of complexes
$$\xymatrix{0\ar[r]&F^{\cdot}\ar[r]^{u}&G^{\cdot}\ar[r]^{v}&H^{\cdot}\ar[r]&0}$$
the diagram:
$$\xymatrix{f_{X}(\text{L}g^{*}F^{\cdot})\otimes f_{X}(\text{L}g^{*}H^{\cdot})\ar[d]^{\eta \cdot\eta}_{\wr}\ar[rr]^(.60){i_{Y}(\text{L}g^{*}(u,v))}_(.60){\sim}&& f_{X}(\text{L}g^{*}G^{\cdot})\ar[d]^{\eta}_{\wr}\\
g^{*}f_{Y}(F^{\cdot})\otimes g^{*}f_{Y}(H^{\cdot})\ar[rr]^{i_{Y}(u,v)}_{\sim}& &g^{*}f_{Y}(F^{\cdot})
}$$
commutes, where $\text{L}g^{*}$ is the left derived functor of the morphism $g$ and exists for the category whose objects are short exact sequences of 
complexes of three objects in $\text{Mod}(Y)$ and whose morphisms are triples
such that the resulting diagram (like the diagram in i) but not isomorphism in general) commutes (see \cite{Kund}, Prop. 3). Moreover if 
$$\xymatrix{X\ar[r]^{g}&Y\ar[r]^{h}&Z}$$ are two consecutive morphisms, the diagram:

$$\xymatrix{ f_{X}(\text{L}g^{*}\text{L}h^{*})\ar[r]^{\eta(g)}_{\sim}\ar[d]^{f_{X}(\theta)}_{\wr}& g^{*}f_{Y}\text{L}h^{*}\ar[r]^{g^{*}\eta(h)}_{\sim}&g^{*}h^{*}f_{Z}\ar[d]_{\wr}\\
          f_{X}(\text{L}(h\cdot g)^{*})\ar[rr]^{\sim}&  &(h\cdot g)^{*}f_{Z}
}$$
commutes where $\theta$ is the canonical isomorphism 
$$\theta:\xymatrix{ \text{L}g^{*}\text{L}h^{*}\ar[r]^{\sim}&\text{L}(h\cdot g)^{*}},$$
iv) On finite complexes of locally free $\mathcal{O}_{X}$-modules
$$f=\det~\text{and}~~~i=i.$$
\end{defi}
\begin{theo}\label{ucd}
 There is one, and, up to canonical isomorphism, only one extended determinant functor $(f,\text{i})$ which we will write $(\det,\text{i})$.
\end{theo}
\begin{proof}
See \cite{Kund}, Theorem 2, Pag. 42.
\end{proof}

The theorem above implies that the functor $(\det, i)$ have same compatibility as ordinary $\det^{*}$. In particular:

a) If each term $\mathcal{F}^{n}$ in the corresponding perfect complex $\mathcal{F}^{\cdot}$ is itself perfect, i.e., has locally a finite free resolution, then $$\det(\mathcal{F}^{\cdot})\cong \otimes_{n}\text{det}^{*}(\mathcal{F}^{n})^{(-1)^{n}}.$$

b) If the cohomology sheaves $H^{n}(\mathcal{F}^{\cdot})$ of the complex are perfect we denote the objects of subcategory by $\text{Parf}^{0}\subset\text{Parf}$,
then $$\det(\mathcal{F}^{\cdot})\cong \otimes_{n}\text{det}^{*}(H^{n}(\mathcal{F}^{\cdot}))^{(-1)^{n}}.$$
From the previous theorem and a) and b), we have the following corollary:
\begin{coro}
 Let $$\xymatrix{\ar[r]&\mathcal{F}_{1}^{\cdot}\ar[r]^{u}&\mathcal{F}_{2}^{\cdot}\ar[r]^{v}&\mathcal{F}_{3}^{\cdot}\ar[r]^{w}&T\mathcal{F}_{1}^{\cdot}\ar[r]& }$$
be a triangle in $\text{Parf}_{X}$ such that $\mathcal{F}_{i}^{.}\in\text{Parf}^{0}_{X}$. We have a unique isomorphism,
$$\xymatrix{i_{X}(u,v,w):\det(\mathcal{F}_{1}^{\cdot})\otimes\det(\mathcal{F}_{3}^{\cdot})\ar[r]^(.70){\sim}&\det(\mathcal{F}_{2}^{\cdot}) }$$
which is functorial with respect to such triangles, i.e., the extended functor can be defined for triangles instead of short exact sequences of complexes.
\end{coro}
\begin{proof}
 See \cite{Kund}, Page 43.
\end{proof}
\begin{rem}\label{pcd}
 From the previous corollary, it is to say that the extended determinant functor can be defined on triangle in $\text{Parf}$ which satisfies similar properties of short exact sequences from i) to iv) if the objects in the triangles
 are in $\text{Parf}^{0}$.  For a vector bundle $E$, if $Rf_{*}E$ is a strictly perfect complex under some suitable morphism where $Rf_{*}$ is viewed as the right derived functor of $f_{*}$, then
 the properties of extended determinant functor is valid for the strictly perfect complex.
\end{rem}

To conclude this section, we put the determinant functor, the virtual category, and the Picard category together to make the following definition.
\begin{defi}\label{vdk}
For a scheme $X$, we denote by $Vect(X)$ the exact category of vector bundles over $X$ and by $V(X): =V(Vect(X))$ (resp. $V(Y)$) the virtual category of vector bundles on $X$ (resp. $Y$).
Let $f:X\rightarrow Y$ be a smooth and projective morphism and  $Y$ carries an ample invertible sheaf. Then there is an induced functor from $V(X)$ to the Picard category $\mathscr{P}is_{Y}$ (the definition of $\mathscr{P}is_{Y}$is in the example \ref{lbd}) denoted 
by $\det Rf_{*}$, which is defined as follows:

In the Theorem \ref{ucd}, we have a unique functor $\det: \text{Perf-is}_{Y} \rightarrow \mathscr{P}is_{Y}$ . For any vector bundle $E$ from the exact category $Vect(X)$, it can be viewed a perfect complex $E^{.}$ with a term $E$ at 
degree $0$ and $0$ at other degree.   
For any perfect complex $E^{.}$ of $\mathcal{O}_{X}$-modules and the morphism $f$ in the definition, $Rf_{*}E^{.}$ is still a perfect complex of $\mathcal{O}_{Y}$-modules (see \cite{Berth}, IV, 2.12). Therefore, we have a functor
$\det Rf_{*} :(Vect(X), iso)\rightarrow\text{Perf-is}_{Y}$ where $(Vect(X), iso)$ is the category with the same objects from $Vect(X)$ and morphisms being only isomorphisms.
By the definition of the extended determinant functor $\det$, it can be verified that $\det Rf_{*}$ satisfies the same conditions from a) to d) with $[~]$ in Def. \ref{vircon}.
By the universality of the virtual category $V(X)$, the functor $\det Rf_{*}$ factors uniquely up to unique isomorphism through $(Vect(X), iso)\rightarrow V(X)$. More clearly, we have the following diagram:
$$\xymatrix{(Vect(X), iso)\ar[d]\ar[r]^(.60){Rf_{*}}&\text{Perf-is}_{Y}\ar[d]^{\det}\\
  V(X)\ar@{.>}[r]^{\det Rf_{*}}&\mathscr{P}is_{Y}     
        }$$
Meanwhile, there is a functor from
 $V(X)\rightarrow\mathscr{P}is_{Y}$ which is still denoted by $\det Rf_{*}$.
 \end{defi}

\section{The Adams-Riemann-Roch Theorem}
\subsection{The Frobenius morphism}
As in the introduction, for the $p$-th Adams-Riemann-Roch theorem in the case of characteristic $p>0$, the Frobenius morphism for the schemes of characteristic $p$ plays a key role in the proof of \cite{Pi}. 
For more information about Frobenius morphisms, see the electronic lectures of Professor Lars Kindler or Qing Liu's book \cite{Liu}. We say that a scheme $X$ is of characteristic $p$, if we have $p \mathcal{O}_{X}=0$.

(1) For every $\mathbb{F}_{p}-$algebra $A$, we have the classical Frobenius endomorphism
                                    $$Frob_{A}: A\rightarrow A$$
                                         $$\quad~~~~~~~ a\mapsto a^{p}.$$
Hence, for every affine $\mathbb{F}_{p}$-scheme $X=\text{Spec}~A$, we have an affine Frobenius morphism                                   $$F_{X}:X\rightarrow X.$$
We also see that $F_{X}$ is the identity morphism on the underlying topological space Sp$(X)$ because if $a^{p}\in\mathfrak{p}$ for any prime ideal $\mathfrak{p}$ of $A$, we have $a\in\mathfrak{p}$. We denote the corresponding  morphism over sheaves
by $Frob_{\mathcal{O}_{X}}$ such that for every open set $U$ of $X$ we have $Frob_{\mathcal{O}_{X}}(U)= Frob_{\mathcal{O}_{X(U)}}:\mathcal{O}_{X}(U)\rightarrow\mathcal{O}_{X}(U)$. To be precise, the affine Frobenius morphism
is the morphism $F_{X}:=(Id_{sp(X)},Frob_{\mathcal{O}_{X}})$. 

For any $\mathbb{F}_{p}$-algebra homomorphism $f:A\rightarrow B$, the following diagram commutes

 $$\xymatrix{A\ar[d]^{f}\ar[r]^{Frob_{A}}&A\ar[d]^{f}\\
                    B\ar[r]^{Frob_{B}}&B }$$

Therefore, we can define the absolute $Frobenius$ morphism in general case. 
\begin{defi}
Let $X$ be a $\mathbb{F}_{p}$-scheme, the $\textbf{absolute Frobenius}$ morphism on $X$, denoted by $F_{X}$, is a morphism $$X\longrightarrow X$$ such that for every open affine subset $U$ of $X$ we have $F_{X|U}:U\longrightarrow U$. Hence we have Frobenius morphism
 $F_{X}=(Id _{Sp(X)},\text{Frob}_{\mathcal{O}_{X}})$ for a general $\mathbb{F}_{p}$-scheme.
\end{defi}
Let $S$ be an $\mathbb{F}_{p}$-scheme and $X$ an $S$-scheme. It is easy to verify that the following diagram commutes
$$\xymatrix{X\ar[d]_{f}\ar[r]^{F_{X}}&X\ar[d]^{f}\\
                    S\ar[r]^{F_{S}}&S }$$
where $f$ is the structure morphism.

From the diagram above, we find that $F_{X}$ is not an $S$-morphism. But we can get an $S$-morphism by making use of the diagram above and introducing the relative Frobenius morphism.

\begin{defi}\label{rel,fr}
For any morphism $f:X\longrightarrow S$ of $\mathbb{F}_{p}$-schemes, the following diagram commutes
$$\xymatrix{X\ar@/_/[ddr]_{f}\ar@{.>}[dr]|{F_{X/S}}\ar@/^/[rrd]^{F_{X}}&  & \\
                            & X^{'}\ar[r]_{W_{X/S}}\ar[d]^{f^{'}}&X\ar[d]^{f}    \\
                             &S\ar[r]^{F_{S}}&S
                                 }~~~~~(1)$$
where $X^{'}$ is the fiber product of the morphism $f:X\longrightarrow S$ by base extension $F_{S}: S\longrightarrow S$. The morphism $F_{X/S}$ is called the $\textbf{relative Frobenius}$ morphism of the morphism $f$ which exists by the universal property of fiber product of schemes.
\end{defi}
\begin{exam} 
Let $A$ be a $\mathbb{F}_{p}$-algebra, $S=\text{Spec}~A$ an affine scheme, $f(x)=\sum_{k=0}^{n}f_{k}x^{k}$ a polynomial in $A[x]$ and $X=\text{Spec}~A[x]/(f(x))$ an affine $S$-scheme. It is easy to verify that $X^{'}\cong \text{Spec}~A[x]/(f^{'}(x))$, where $f^{'}(x)=\sum_{k=0}^{n}f_{k}^{p}x^{k}$. Therefore
we have a corresponding commutative diagram
$$\xymatrix{A[x]/(f(x))&  & \\
                       & A[x]/(f^{'}(x))\ar[lu]|{\widetilde{F}_{X/S}}&A[x]/(f(x))\ar[l]^{\widetilde{W}_{X/S}}\ar@/_/[llu]_{Frob_{A[x]/(f(x))}}   \\
                             &A\ar@/^/[uul]^{\tilde{f}}\ar[u]_{\tilde{f^{'}}}&A\ar[l]^{Frob_{A}}\ar[u]_{\tilde{f}}
                                 }~~~~~(2)$$
where $\tilde{f}$ (resp. $\tilde{f}^{'}$) send every $a\in A$ to its equivalence $\bar{a}$ in $A[x]/(f(x))$ (resp. $A[x]/(f^{'}(x)$), $\widetilde{W}_{X/S}$ sends the equivalence class of the monomial $ax$ to the equivalence class of the monomial $a^{p}x$ and $\widetilde{F}_{X/S}$ sends the equivalence class of the monomial $ax$ to the equivalence class of $ax^{p}$ in $A[x]/(f(x))$.
\end{exam}

In order to prove the next useful lemma, we state the \'{e}tale coordinates as follows:

(\'{E}tale coordinates) Let $f:X\rightarrow S$ be a smooth morphism and $q$ a point of $X$. Since $\Omega_{X/S}^{1}$ is locally free, there is an open affine neighborhood $U$ of the point $q$ and $x_{1},...,x_{n}\in\mathcal{O}(U)$ such that $x_{i}|_{q}=0$ for all $i$ and such that $dx_{1},...dx_{n}$ generate
$\Omega_{X/S}^{1}|_{U}$ (use Nakayama lemma). These sections define an $S$-morphism $h: U\rightarrow \mathbb{A}_{S}^{n}$. This morphism is \'{e}tale
by construction:

Because of smoothness we have the exact sequence of $\mathcal{O}(U)-$modules

$$0\longrightarrow h^{*}\Omega_{Z/S}^{1}\longrightarrow\Omega_{U/S}^{1}\longrightarrow\Omega_{U/Z}^{1}\longrightarrow 0$$
where $Z:=\mathbb{A}_{S}^{n}$ for brevity. By construction $\Omega_{U/Z}^{1}=0$, so $h:U\rightarrow Z$ is smooth and unramified, hence \'{e}tale. The $x_{1},...x_{n}$ are called \'{e}tale coordinates around $q$.

\begin{defi} Let $X,S$ be schemes and $f:X\rightarrow S$ a morphism of schemes. The morphism $f$ is said to be universally homeomorphism if for every scheme $T$ and for every morphism of schemes $g:T\rightarrow S$, the corresponding base change of $f$ is a homeomorphism.
\end{defi}
\begin{exam}\label{u,h}
 It is not difficult to check that the relative Frobenius morphism and the absolute Frobenius morphism are both  universally homeomorphisms since they are already homeomorphisms on topological spaces by verifying their definitions (Actually, they are identities on topological spaces).
\end{exam}
Based on the notations in Definition \ref{rel,fr}, we can state and prove the following important lemma.

\begin{lem}\label{l,f} Let $S$ be a scheme of positive characteristic (say $p$) and $f:X\rightarrow S$ a smooth morphism of pure relative dimension $n$. Then the relative Frobenius $F_{X/S}$ is finite and flat, and the $\mathcal{O}_{X^{'}}-$algebra $(F_{X/S})_{*}\mathcal{O}_{X}$ is locally free of rank $p^{n}$. In particular, if $f$ is \'{e}tale, then $F_{X/S}$ is an isomorphism.
\end{lem}
\begin{proof} We firstly show that $F_{X/S}$ is an isomorphism when $f$ is \'{e}tale, i.e., smooth of relative dimension $0$. In the diagram (1) (we still use the notation of Definition \ref{rel,fr} in the entire proof), we find that $f^{'}$ is
 \'{e}tale and that $f=f^{'}\circ F_{X/S}$ is \'{e}tale, so $F_{X/S}$ is \'{e}tale(by the properties of base change and composition of \'{e}tale morphism). But in the diagram (1), actually $W_{X/S},F_{X}$ induce identity on topological spaces and $F_{X}=W_{X/S}\circ F_{X/S}$, therefore the relative Frobenius morphism also induces an identity on topological spaces.
 In addition, $F_{X/S}$ is an open immersion (in [SGAI.5.1], it is proved that any morphism of finite type is open immersion if and only if the morphism is \'{e}tale and radical. Radical means universally injective which is similar with  universally closed morphism, i.e., injective itself and also injective for any base extension. But universally homeomorphism is 
 equivalent to integral, surjective and universally injective, 
and Example \ref{u,h} said that $F_{X/S}$ is a universal homeomorphism.).
 Putting homeomorphism and open immersion together, we have that $F_{X/S}$ is an isomorphism.

 For $Z=\mathbb{A}^{n}_{S}$ and the projection $f: Z \rightarrow S$, we have a Cartesian diagram 
 $$\xymatrix{Z\ar@/_/[ddr]_{f}\ar@{.>}[dr]|{F_{Z/S}}\ar@/^/[rrd]^{F_{Z}}&  & \\
                            & Z^{'}\ar[r]_{W_{Z/S}}\ar[d]^{f^{'}}&Z\ar[d]^{f}    \\
                             &S\ar[r]^{F_{S}}&S
                                 }$$
 
 as in Definition \ref{rel,fr} and the topological space $\text{Sp}(Z)=\text{Sp}(Z^{'})$. By the definition of $Z$, we have $\mathcal{O}_{Z}=\mathcal{O}_{S}[t_{1},...,t_{n}]$. So the sheaf morphism $F^{\sharp}_{Z/S}:\mathcal{O}_{Z^{'}}\rightarrow (F_{Z/S})_{*}{\mathcal{O}}_{Z}$ is the map

     $$F^{\sharp}_{Z/S}:\mathcal{O}_{S}[t_{1},...,t_{n}]\rightarrow \mathcal{O}_{S}[t_{1},...,t_{n}]$$
     $$t_{i}\mapsto t_{i}^{p}.$$
Hence the monomials $\prod_{i=1}^{n}t_{i}^{k_{i}}$, with $k_{i}$ integers such that $0\leq k_{i}\leq p-1$, form a basis of $(F_{Z/S})_{*}{\mathcal{O}}_{Z}$ over $\mathcal{O}_{Z^{'}}$. Therefore $F_{Z/S}$ is indeed finite locally free of rank $p^{n}$, hence also flat.

For the general case, we can assume by the smoothness of $f$ that locally on $X$, we have the factorization $\xymatrix{X\ar[r]^{h}& Z\ar[r]^{g}& S}$, where $g$ is a projection $Z=\mathbb{A}^{n}_{S}\rightarrow S$ and $h$ is \'{e}tale. We have the following diagram
 $$\xymatrix{X\ar[rr]^{F_{X/S}}\ar[dr]^{F_{X/Z}}\ar[dd]_{h}& &X^{'}=X\times_{F_{S}}S\ar[r]\ar[dd]&X\ar[dd]^{h}\\
                        &X\times_{F_{Z}}Z\ar[ld]\ar[ur]^{\psi}&  &   \\
   Z\ar[rr]^{F_{Z/S}}\ar[drr]_{g}& &Z^{'}=S\times_{F_{S}}Z\ar[r]\ar[d]&Z\ar[d]^{g}\\
                     & & S\ar[r]^{F_{S}}& S                   } (2)$$
In the diagram $(2)$, the right-most part ($X^{'},Z^{'}$, i.e., the fiber products of schemes) is Cartesian. Meanwhile, we have $F_{X/S}=\psi\circ F_{X/Z}$. By the first part, we know that $F_{X/Z}$ is isomorphism because $h$ is \'{e}tale. $\psi$ is finite locally free of rank $p^{n}$ which is from the base 
change of $F_{Z/S}$ when $F_{Z/S}$ is finite locally free of rank $p^{n}$ for $Z=\mathbb{A}^{n}_{S}\rightarrow S$. 
Therefore, $F_{X/S}$ is finite locally free of rank of $p^{n}$, and flat also.
\end{proof}
\begin{exam} Let $X$, $S$ be affine, denoted by $X=\text{Spec}~B$, $S=\text{Spec}~A$, where $A$ is of characteristic $p>0$. Then $X^{'}=B\otimes_{F_{A}}A$, i.e., $ab\otimes 1=b\otimes a^{p}$, and the relative $Frobenius~X\rightarrow X^{'}$ is given by $a\otimes b\mapsto ab^{p}$.

If $B=A[t]$, then $A\otimes_{F_{A}}A[t]\cong A[t]$, $W_{X/S}$ is given by $at\rightarrow a^{p}t$ and $F_{X/S}$ by $at\rightarrow at^{p}$ for $a\in A$ ($W_{X/S}$ is the morphism from $X^{'}$ to $X$ as in Definition \ref{rel,fr}). Hence the image of $F_{A[t]/A}$ is $A[t^{p}]\subseteq A[t]$ and $A[t]$ is freely generated by $t^{i}$, $i=0,...,p$ over $A[t^{p}]$.
\end{exam}

\subsection{The construction of the Bott element}
\begin{defi}\label{ad,2}
For any integer $k\geq 2$, the symbol $\theta^{k}$ refers to an operation, which associates an element of $K_{0}(X)$ to any locally free coherent sheaf on a quasi-compact scheme $X$. It satisfies the following three properties: 

(1) For any invertible sheaf $L$ over $X$, we have

$$\theta^{k}(L)=1+L+\cdots +L^{\otimes k-1};$$

(2) For any short exact sequence $0\longrightarrow E^{'}\longrightarrow E\longrightarrow E^{''}\longrightarrow 0$ of locally free coherent
sheaves on $X$ we have

         $$\theta^{k}(E)=\theta^{k}(E^{'})\otimes\theta^{k} (E^{''});$$

(3) For any morphism of schemes $g : X^{'}\longrightarrow X$ and any locally free coherent sheaf $E$ over $X$ we have

                  $$g^{*}(\theta^{k}(E))=\theta^{k}(g^{*}(E)).$$

If $E$ is a locally free coherent sheaf on a quasi-compact scheme $X$, then the element $\theta^{k}(E)$ is often called the $k$-th Bott element.
\end{defi}

\begin{pro}
The operation $\theta^{k}$ which satisfies three properties above can be defined uniquely.
\end{pro}
\begin{proof}
See \cite{Man}. Lemma 16.2. Subsection 16, or SGA, VII.
\end{proof}

On a quasi-compact scheme of characteristic $p$, Pink and R\"{o}ssler constructed an explicit representative of the $p$-th Bott element (see \cite{Pi}, Sect. 2). 

We recall the construction:

Let $p$ be a prime number and $Z$ a scheme of characteristic $p$. Let $E$ be a locally free coherent sheaf $Z$. For any integer $k\geq 0$ let
$\text{Sym}^{k}(E)$ denote the $k$-th symmetric power of $E$. Then 
$$\text{Sym}(E):=\bigoplus _{k\geq 0}\text{Sym}^{k}(E)$$ is quasi-coherent graded $\mathcal{O}_{Z}$-algebra, called the symmetric algebra of $E$. Let
$\mathcal{J}_{E}$ denote the graded sheaf of ideals of $\text{Sym}(E)$ that is locally generated by the sections $e^{p}$ of $\text{Sym}^{p}(E)$ for
all sections $e$ of $E$, and set $$\tau(E):=\text{Sym}(E)/\mathcal{J}_{E}.$$
Locally this construction means the following. Consider an open subset $U\subset Z$  such that $E|_{U}$ is free, and
choose a basis $e_{1},\ldots, e_{r}$. Then $\text{Sym}(E)|_{U}$ is the polynomial algebra over $\mathcal{O}_{Z}$ in the variables $e_{1},\ldots,e_{r}$.
Since $Z$ has characteristic $p$, for any open subset $V\subset U$ and any sections $a_{1},\ldots,a_{r} \in \mathcal{O}_{Z}(V)$ we have 
$$(a_{1}e_{1}+\ldots +a_{r}e_{r})^{p}=a_{1}^{p}e_{1}^{p}+\ldots +a_{r}^{p}e_{r}^{p}.$$
It follows that $\mathcal{J}_{E}|_{U}$ is the sheaf of ideals of $\text{Sym}(E)|_{U}$ that  is generated by $e_{1}^{p}\ldots,e_{r}^{p}$. Clearly that description is independent of
the choice of basis and compatible with localization; hence it can be used as an equivalent definition of $\mathcal{J}_{E}$ and $\tau(E)$.
The local description also implies that $\tau(E)|_{U}$ is free over $\mathcal{Z}|_{U}$ with the basis the images of the monomials $e_{1}^{i_{1}}\cdots e_{r}^{i_{r}}$ for all choices of exponents $0\leqq i_{j}< p$. 

It can be showed that $\tau(E)$ satisfies the defining properties of the $p$-th Bott element. In other words, we have the following proposition (see \cite{Pi}, Prop. 2.6).

\begin{pro}\label{btp}
For any locally free coherent sheaf $E$ on a quasi-compact scheme $Z$ of characteristic $p>0$, we have $\tau(E)=\theta ^{p}(E)$ in $K_{0}(Z)$.
\end{pro}
\subsection{The Adams Riemann Roch Theorem in positive characteristic}\label{section ARR char p}
In order to state the $p$-th  Adams-Riemann-Roch theorem in positive characteristic in \cite{Pi}, we firstly give the conditions of the Adams-Riemann-Roch theorem.  We assume that $f:X\rightarrow Y$ is a projective local complete intersection morphism and $Y$ is quasi-compact scheme 
and carries an ample invertible sheaf. Furthermore, we make the supplementary 
hypothesis that $f$ is smooth and that $Y$ is a scheme of
characteristic $p>0$.  Let $\Omega_{f}$, the relative differentials of $f$, be rank $r$, which is a locally constant function. 

Based on the condition above, we draw the diagram again, i.e., the commutative diagram
$$\xymatrix{X\ar@/_/[ddr]_{f}\ar@{.>}[dr]|{F_{X/Y}}\ar@/^/[rrd]^{F_{X}}&  & \\
                            & X^{'}\ar[r]_{J}\ar[d]^{f^{'}}&X\ar[d]^{f}    \\
                             &Y\ar[r]^{F_{Y}}&Y
                                 }~~~~~(4)$$
 where $ F_{X}$ and $F_{Y}$ are obvious absolute Frobenius morphisms respectively and the square is Cartesian. We also denote the relative Frobenius morphism of the morphism $f$ by $F:=F_{X/Y}$ for
simplicity. In the following propositions and proofs until the end of the note, we will use these notations.

Since the pull-back $F^{*}$ is adjoint to $F_{*}$ (see \cite{Hart}, Page 110), there is a natural morphism of $\mathcal{O}_{X}$-algebras $F^{*}F_{*}\mathcal{O}_{X}\rightarrow\mathcal{O}_{X}$. Let $I$
be the kernel of the natural morphism, which is a sheaf of the ideal of $F^{*}F_{*}\mathcal{O}_{X}$ by the definition. In \cite{Pi}, the following definition is made
$$Gr(F^{*}F_{*}\mathcal{O}_{X}):=\bigoplus_ {k\geq 0}I^{k}/ I^{k+1}$$
which denote the associated graded sheaf of $\mathcal{O}_{X}$-algebras. Let $\Omega_{f}$ be the sheaf of relative differentials of $f$.
Also, they proved the following key proposition which can be used to prove the $p$-th Adams-Riemann-Roch theorem in positive characteristic (see \cite{Pi}, Prop. 3.2).
\begin{pro}\label{gci}
There is a natural isomorphism of $\mathcal{O}_{X}$-modules
$$ I/ I^{2}\cong \Omega_{f}$$
and a natural isomorphism of graded  $\mathcal{O}_{X}$-algebras
$$\tau(I/ I^{2})\cong Gr(F^{*}F_{*}\mathcal{O}_{X})$$
 \end{pro}

According to the proposition above, directly there are isomorphisms $$Gr(F^{*}F_{*}\mathcal{O}_{X})\cong\tau(I/ I^{2})\cong \tau(\Omega_{f}).$$ 
Moreover, Proposition \ref{btp} also
implies $\tau(\Omega_{f})\cong \theta ^{p}(\Omega_{f})$. In the Grothendieck groups, we have $Gr(F^{*}F_{*}\mathcal{O}_{X})\cong F^{*}F_{*}\mathcal{O}_{X}$, then the equality $F^{*}F_{*}\mathcal{O}_{X}\cong \theta ^{p}(\Omega_{f})$ holds by viewing them as elements of the Grothendieck groups.

In the case of characteristic $p>0$, the Adams operation $\psi^{p}$ can be described, especially. We will give the proposition after recalling the definition of the Adams operation.
\begin{defi}\label{ad,1}
Let $X$ be a quasi-compact scheme. For any positive integer $k\geq 2$, the $k$-th Adams operation is the functorial endomorphism $\psi^{k}$ of unitary ring $K_{0}(X)$ which is uniquely determined by the following two conditions,

(1) $\psi^{k}f^{*}=f^{*}\psi^{k}$ for any morphism of Noetherian schemes $f:X\longrightarrow Y$.

(2) For any invertible sheaf $L$ over $X$, $\psi^{k}(L)=L^{\otimes k}$.
\end{defi}
A more interesting proposition related to the Frobenius morphism and the Adams operation is the following:
\begin{pro}\label{frebad}
  For a scheme $Z$ of characteristic $p>0$ and its absolute Frobenius morphism $F_{Z}: Z\rightarrow Z$,  we claim that the pullback $F_{Z}^{*}: K_{0}(Z)\rightarrow K_{0}(Z)$ is  just the $p$-th Adams operation $\psi^{p}$. 
\end{pro}
\begin{proof}
 This is a well-known fact (see \cite{Bern 1}, Pag. 64, Proposition 2.15), which is also a consequence of the splitting principle (see \cite{Man}, Par. 5).
\end{proof}
\begin{theo}
The Adams-Riemann-Roch theorem is true under the assumption given in the beginning of this subsection, i.e., the following equality 
$$\psi^{p}(R^{\bullet}f_{*}(E))=R^{\bullet}f_{*}(\theta ^{p}(\Omega_{f})^{-1}\otimes\psi^{p}(E)) $$
holds in $K_{0}(Y)[\frac{1}{p}]:=K_{0}(Y)\otimes_{\mathbb{Z}}\mathbb{Z}[\frac{1}{p}]$.
\end{theo}
\begin{proof} See \cite{Pi}, page 1074. 
\end{proof}
\begin{rem}
 Based on the condition we gave in the beginning of this subsection,  the equality $$\psi^{k}(R^{\bullet}f_{*}(E))=R^{\bullet}f_{*}(\theta ^{k}(\Omega_{f})^{-1}\otimes\psi^{k}(E))$$ for any integer $k\geq 2$, also holds
in $K_{0}(Y)\otimes \mathbb{Z}[\frac{1}{k}]$ not only for $k=p$ (see \cite{FL}, V. Th. 7.6), but its proof is more complicated than \cite{Pi}.
\end{rem}

\section{A functorial Riemann-Roch theorem in positive characteristic}
\subsection{The Deligne pairing}\label{sdp}
Before stating the Deligne's functorial Riemann-Roch theorem, it is necessary to introduce the Deligne pairing, which appeared in [3] for the first time and  was extended 
to a more general situation by S. Zhang (see [8]). We shall only need  the basic definition here.

Let $g: S^{'}\rightarrow S$ be a finite and flat morphism. Then $g_{*}\mathcal{O}_{S^{'}}$ is a locally free sheaf of constant rank (say $d$). 
We have a natural morphism of $\mathcal{O}_{S}$-modules $$g_{*}\mathcal{O}_{S^{'}}\rightarrow End_{\mathcal{O}_{S}}(\mathcal{O}_{S^{'}}).$$
Taking the composition with the determinant, we obtain a morphism $$\text{N}: g_{*}\mathcal{O}_{S^{'}}\rightarrow \mathcal{O}_{S}$$
Generally, we have the following definition (see \cite{Del}, Sect. 7.):

\begin{defi}\label{d,n}

Let $g: S^{'}\rightarrow S$ be finite and flat.  For any invertible sheaf $L$ over $S^{'}$, and the sub-sheaf $L^{*}$
of invertible sections of $L$, its norm $\text{N}_{S^{'}/S}(L)$ is defined to be an invertible sheaf $N$ over $S$ equipped with a morphism of sheaves $\text{N}_{S^{'}/S}: 
g_{*}L^{*}\rightarrow N^{*}$ satisfying $\text{N}_{S^{'}/S}(u\ell)=\text{N}(u)\text{N}_{S^{'}/S}(\ell)$ for any local sections $u$ of $g_{*}\mathcal{O}_{S^{'}}$ 
and $\ell$ of $g_{*}L^{*}$. 

In other words, the norm morphism induces a norm functor $\text{N}_{g}$ from the category of line bundles over $S^{'}$ to the category of line bundles over
$S$ together with a collection of homomorphisms $\text{N}_{g}^{L}: g_{*}L\rightarrow \text{N}_{g}(L)$ of sheaves of sets, for all line bundle $L$
over $S^{'}$, functorial under isomorphisms of line bundles over $S^{'}$, sending local generating sections over $S^{'}$ to the local generating 
sections over $S$ and such that the equality 
$\text{N}_{g}^{L}(xl)=\text{N}(x)\text{N}_{g}^{L}(l)$ holds for all local sections $x$ of $g_{*}\mathcal{O}_{S^{'}}$ and $l$ of $g_{*}L$. Moreover, the functor 
$\text{N}_{g}$ together with the collection of the $\text{N}_{g}^{L}$ is unique up to unique isomorphism. 

The norm functor is a special case of the trace of a torsor for a commutative group scheme under a finite flat morphism (see \cite{Dir}, expos\'e XVII, 6.3.26).
Instead of $\text{N}_{g}$, we also write $\text{N}_{S^{'}/S}$ for the norm functor when the morphism is clear in the specific context. We list the basic properties of the norm functor as follows:

\begin{pro}
 The norm functor has the following properties:
 
 (1) The functor $\text{N}_{S^{'}/S}$ is compatible with any base change $Y\rightarrow S$;
 
 (2) If $L_{1}$ and $L_{2}$ are two line bundles on $S^{'}$, there is a natural isomorphism 
 $$\text{N}_{S^{'}/S}(L_{1}\otimes_{\mathcal{O}_{S^{'}}} L_{2})\cong \text{N}_{S^{'}/S}(L_{1})\otimes_{\mathcal{O}_{S}} \text{N}_{S^{'}/S}(L_{2});$$

(3) If $S_{1}\rightarrow S_{2}\rightarrow S_{3}$ are finite and flat morphisms, there is a natural isomorphism
$$\text{N}_{S_{1}/S_{3}}\cong \text{N}_{S_{2}/S_{3}}\circ\text{N}_{S_{1}/S_{2}};$$

(4) There is a functorial isomorphism $$\text{N}_{S^{'}/S}(L)\cong \text{Hom}_{\mathcal{O}_{S}}(det_{\mathcal{O}_{S}} g_{*}\mathcal{O}_{S^{'}}, det_{\mathcal{O}_{S}} g_{*}L).$$
 \end{pro}
\begin{proof}
 See \cite{Dir}, expos\'e XVII, 6.3.26 for (1), (2), and (3). See \cite{Dir}, expos\'e XVIII, 1.3.17. for (4).
\end{proof}
 
From the properties above, it is to say that the pair $(N,\text{N}_{S^{'}/S})$ is unique in the sense of unique isomorphism up to a sign.
\end{defi}

For any locally free coherent sheaves $F_{0}$ and $F_{1}$ with same rank over $S^{'}$ and the morphism $g: S^{'}\rightarrow S$ as in the definition of the norm, we have a canonical isomorphism 
$$\det(g_{*}F_{0}-g_{*}F_{1})=\text{N}_{S^{'}/S}\det(F_{0}-F_{1})~~~~~(a) $$
$$ \text{i.e.,}~~~~~~~\det g_{*}F_{0}\otimes(\det g_{*}F_{1})^{-1}=\text{N}_{S^{'}/S}(\det F_{0}\otimes (\det F_{1})^{-1} )$$
by viewing $F_{0}$ and $F_{1}$ as the virtual objects in Def. \ref{vdk}, which is compatible with localization over $S$ and is characterized by the fact that the trivialization have a corresponding isomorphism for any isomorphism $u:F_{1}\rightarrow F_{0}$. 
Such an isomorphism  $u$ exist locally on $S$ and it doesn't depend on the choice of $u$. The fact that the isomorphism (a) doesn't depend on the choice of $u$ is because for an automorphism $v$ of $F_{1}$, we have $\det (v, g_{*}F_{1})=\text{N}_{S^{'}/S}\det (v, F_{1})$. 

For more links between the functor $\det$ and the norm functor, see \cite{Del}, Pag. 146-147.

After giving the definition of the norm, we can define the Deligne pairing as follows:
\begin{defi}\label{d,p.n}(see \cite{Del}, \S 6.1)
Let $f:X\rightarrow S$ be a proper, flat morphism and of purely relative dimension $1$.
Let $L,M$ be two line bundles on $X$. Then $\langle L,M\rangle$ is defined to be the $\mathcal{O}_{S}$-module
which is generated, locally for Zariski topology on $S$, by the symbols $\langle\ell,m\rangle$ for sections $\ell, m$ of $L,M$ respectively with the following relations 
$$\langle\ell,gm\rangle=g(\text{div}(\ell))\langle\ell,m\rangle$$
$$\langle g\ell,m\rangle=g(\text{div}(m))\langle\ell,m\rangle$$
where $g$ is a non-zero section of $\mathcal{O}_{X}$ and $g(\text{div}(\ell))$, $g(\text{div}(m))$ are interpreted as a norm: for a relative Cartier divisor $D$ on $X$, i.e., $D\rightarrow S$ is finite and flat in our case (see \cite{Kat}, Chapt. 1, \S 1.2 about relative Cartier divisors ), we put 
$g(D):=\text{N}_{D/S}(g)$, then we have $g(D_{1}+D_{2})=g(D_{1})\cdot g(D_{2})$. If $\text{div}(\ell)=D_{1}-D_{2}$, we put 
$g(\text{div}(\ell))=g(D_{1})\cdot g(D_{2})^{-1}$. One checks that this is independent of the choice $D_{1}, D_{2}$. Moreover, the $\mathcal{O}_{S}$-module
$\langle L,M\rangle$ is also a line bundle on $S$.
\end{defi}
For $L=\mathcal{O}(D)$, and the canonical section $1$ of $\mathcal{O}(D)$, we have $\langle 1,fm\rangle=\text{N}_{D/S}(f)\cdot\langle 1,m\rangle$ for non-zero section $f$ of $\mathcal{O}_{X}$ and a section $m$ of the line bundle $M$ on $X$.

Let $g:X\rightarrow S$ be proper, flat and purely of relative dimension $1$ and $D$ be a relative Cartier divisor of $g$. For any invertible sheaf $M$ over $X$, we have 
$$\xymatrix{\langle\mathcal{O}(D), M\rangle \ar[r]^{\sim}& \text{N}_{D/S}}(M) : ~\langle 1,m\rangle\mapsto\text{N}_{D/S}(m).$$
In other words, for any invertible sheaf $L$ over $X$ and for any section $\ell$ of $L$, which is not a zero divisor in every fiber, we have \\
$$\xymatrix{\langle L,M\rangle \ar[r]^(.40){\sim}&\text{N}_{\text{div}(\ell)/S}(M) } : ~\langle \ell,m\rangle\mapsto\text{N}_{\text{div}(\ell)/S}(m).$$

From Definition \ref{d,p.n}, we have a bi-multiplicative isomorphism:
$$\langle L_{1}\otimes L_{2}, M\rangle\cong\langle L_{1},M\rangle\otimes\langle L_{2},M\rangle$$
$$\langle L,M_{1}\otimes M_{2} \rangle\cong\langle  L,M_{1}\rangle\otimes\langle  L,M_{2}\rangle$$
and symmetric isomorphism $$\langle L,M\rangle\cong\langle M,L\rangle.$$ when $L=M$ the symmetric isomorphism is obtained by multiplication by $(-1)^{\text{deg}L}$
(see SGA4, XVIII 1.3.16.6).

\subsection{ A functorial Riemann-Roch theorem in positive characteristic}
In this subsection, as in section 2, the Picard category of graded line bundles still will be denoted by $\mathscr{P}is_{X}$ and the virtual category of the exact category of vector bundles will be denoted by $V(X)$ for any scheme $X$.
For any vector bundle $E$ from an exact category of vector bundles, which is viewed as a complex, $Rf_{*}E$ is a complex again under some given morphism $f$.

About the Deligne pairing, the most important proposition we will use is the following:
\begin{pro}\label{d,p}
Let $f:X\rightarrow S$ be proper, flat and purely of relative dimension $1$. Let $E_{0}$ and $E_{1}$ be locally free coherent sheaves with same rank everywhere over $X$,
$F_{0}$ and $F_{1}$ with the same property as $E_{0}$ and $E_{1}$. Then we have the following isomorphism
$$\langle\det(E_{0}-E_{1}),\det(F_{0}-F_{1})\rangle\cong\det Rf_{*}((E_{0}-E_{1})\otimes(F_{0}-F_{1})).$$
\end{pro}
\begin{proof} The key point is to verify that
the determinant functor under the Deligne pairing and $\det Rf_{*}$ are compatible with additivity, respectively. Furthermore, their local trivialization are simultaneously identified with
the corresponding norm functor. This is the construction 7.2 of [3]. The precise proof is in [3] (see [3], Pag. 147-149). 
\end{proof}
\begin{coro}\label{trs} 
In particular, we have a canonical isomorphism $\det Rf_{*}((H_{0}-H_{1})^{\otimes l}\otimes H)\cong\mathcal{O}_{S}$ if $l\geq 3$ and the ranks
rk$H_{0}$ = rk$H_{1}$, for any vector bundles $H_{0}, H_{1}, H$ over $X$ and $f$ as in the theorem, which is stable under base change.
\end{coro}
\begin{proof} 
It suffices to prove the conclusion for $l=3$ because it is automatic for $l>3$ after proving the conclusion for $l=3$ according to the following proof. We apply Proposition \ref{d,p} to $E_{0}=H_{0}\otimes H$, $E_{1}=H_{1}\otimes H$, and $F_{0}=H_{0}^{\otimes 2}+H_{1}^{\otimes 2},$
 $F_{1}=2(H_{0}\otimes H_{1})$. \\
Then we have the following:
\begin{align}
 (E_{0}-E_{1})\otimes(F_{0}-F_{1})&\cong H\otimes(H_{0}-H_{1})\otimes(H_{0}-H_{1})^{\otimes 2}\notag\\
&\cong (H_{0}-H_{1})^{\otimes 3}\otimes H.\notag 
\end{align}
Because of the ranks $rk H_{0} =rk H_{1}$, we immediately have equalities of ranks $rk E_{0}= rk E_{1}$, $rkF_{0}=rk F_{1}$.
Meanwhile, notice that 
\begin{align}
\det(F_{0}-F_{1})&\cong(\det(H_{0}))^{\otimes 2rk(H_{0})}\otimes(\det(H_{1}))^{\otimes 2rk(H_{1})}\notag\\
&\otimes((\det(H_{0})^{\otimes -rk(H_{0})}\otimes(\det(H_{1}))^{\otimes -rk(H_{1})})^{2}\notag\\
&\cong\mathcal{O}_{X}.\notag
\end{align}
According to the bi-multiplicativity of the Deligne pairing ( see statements after Def. \ref{d,p.n} or see \cite{Del}, \S 6. 6.2), by a trivial computation:
$\langle\mathcal{O}_{X},L\rangle\cong\langle\mathcal{O}_{X}\otimes\mathcal{O}_{X},L\rangle\cong\langle\mathcal{O}_{X},L\rangle\otimes\langle\mathcal{O}_{X},L\rangle$, the obvious consequence is 
$\langle\mathcal{O}_{X},L\rangle\cong\mathcal{O}_{S}$ for any line bundle $L$ over $X$.
Now, we can obtain the corollary by 
\begin{align}
\det Rf_{*}((H_{0}-H_{1})^{\otimes 3}\otimes H)&\cong\det Rf_{*}((E_{0}-E_{1})\otimes (F_{0}-F_{1}))\notag\\
&\cong\langle(\det(F_{0}-F_{1}),\det(E_{0}-E_{1})\rangle\notag\\
&\cong\langle\mathcal{O}_{X},\det(E_{0}-E_{1})\rangle\notag\\
&\cong \mathcal{O}_{S}.\notag
\end{align}

For any morphism $g :S^{'}\rightarrow S$, we have the fiber product under base change:
$$\xymatrix{X^{'}\ar[r]^{g^{'}}\ar[d]_{f^{'}}&X\ar[d]^{f}\\
S^{'}\ar[r]^{g}&S}$$

Furthermore, the proper morphism and the flat morphism are stable under base change (see \cite{Liu}, Chapt. 3), i.e., the morphism $f^{'}$ is flat of relative dimension $1$ and proper. 
For any vector bundle $F$ over $X$, then we have the isomorphism: $$g^{*}(\text{det}_{S}(Rf_{*}F))\cong{det}_{S^{'}}(\text{L}g^{*}(Rf_{*}F))\cong\text{det}_{S^{'}}(Rf^{'}_{*}(g^{'*}F)$$ (we will prove the isomorphisms in (II) of Theorem \ref{functor})

Let $F$ be $(H_{0}-H_{1})^{\otimes l}\otimes H$, and the isomorphism above becomes $$ \text{det}_{S^{'}}(Rf^{'}_{*}(g^{'*}F))\cong g^{*}(\text{det}_{S}(Rf_{*}F))\cong g^{*}(\mathcal{O}_{S})\cong \mathcal{O}_{S^{'}}.$$ So it is done.

\end{proof}
Before proving our theorem, we state the Deligne's functorial Riemann-Roch theorem so that we can make some comparison later. His theorem is true regardless of any characteristic.
\begin{theo}\label{Del,func}(Deligne)
 Suppose that a morphism $f:X\rightarrow S$ is proper and smooth of relative dimension $1$, with geometrically connected fibers. For any line bundle $L$ in $V(X)$, then there exists a unique, up to sign, isomorphism of line bundles
 $$(\det Rf_{*}(L))^{\otimes 12}\cong \langle\omega, \omega\rangle\otimes\langle L, \omega ^{-1}\otimes L\rangle$$ where $\omega:= \Omega_{f}$ is the relative differentials of the morphism $f$. 
\end{theo}
\begin{proof}
See \cite{Del}, Pag. 170, Theorem. 9.8.
\end{proof}
\begin{rem}\label{explai}
According to the property of Deligne pairing 
(Prop. \ref{d,p}), if $u$ and $v$ are virtual vector bundles of rank $0$ over $X$, then there is a canonical isomorphism:
$$\langle\det u,\det v\rangle\cong\det Rf_{*}(u\otimes v).$$ 
In particular, let $u$ be $L-\mathcal{O}$ and $v$ be $M-\mathcal{O}$. Then we have
\begin{align}
\langle L, M\rangle&\cong\langle \det (L -\mathcal{O}), \det(M-\mathcal{O})\rangle\notag\\
&\cong \det Rf_{*}((L -\mathcal{O})\otimes (M-\mathcal{O}))\notag\\
&\cong\det  Rf_{*}(L\otimes M) \cdot(\det  Rf_{*}(L))^{-1}\cdot (\det  Rf_{*}(M))^{-1} \cdot\det Rf_{*}(\mathcal{O})\notag
\end{align}

for line bundles $L, M$ and the trivial bundle $\mathcal{O}$ over $X$.
Similarly, we have the isomorphism 
\begin{align} \langle L&,  ~\omega^{-1} \otimes L \rangle \notag\\
&\cong\det Rf_{*}(L^{2}\otimes \omega^{-1} ) \otimes (\det Rf_{*}L)^{-1}\otimes (\det Rf_{*}L\otimes\omega^{-1})^{-1}\otimes \det Rf_{*}(\mathcal{O}).\notag
\end{align}

Moreover, by Mumford's isomorphism (see \cite{Mum}, \S 5), i.e., 
$(\det Rf_{*}\mathcal{O})^{\otimes 12 }\cong \langle\omega,\omega\rangle$, the equivalent expression of Deligne's isomorphism is  the following:
\begin{align}
(\det &Rf_{*}L)^{\otimes 18 }\notag\\
&\cong (\det  Rf_{*}\mathcal{O})^{\otimes 18 }\otimes (\det  Rf_{*}(L^{\otimes 2}\otimes 
\omega^{-1}))^{\otimes 6}\otimes  (\det  Rf_{*}(L\otimes \omega^{-1}))^{\otimes (-6)}.\notag
\end{align} 
which is the statement appearing in the introduction.
\end{rem}
After preparing well all we will need, our main result is as follows:
\begin{theo}\label{functor}
Let $f:X\rightarrow S$ be projective and smooth of relative dimension $1$, where $S$ is a quasi-compact scheme of characteristic $p>0$ and carries an ample invertible sheaf.
 Let $L$ be a line bundle over $X$ and $\omega:= \Omega_{f}$ be the sheaf of relative differentials of $f$, then we have
\begin{align}
 (I)~(\det Rf_{*}L)^{\otimes p^{4}}\cong&(\det Rf_{*}L^{\otimes p})^{\otimes~3p^{2}-3p+1}\otimes_{k=1}^{p-1}(\det Rf_{*}(L^{\otimes p}\otimes \omega^{k}))^{\otimes k+1-3p}\notag\\
&\otimes_{k=0}^{p-2}(\det Rf_{*}(L^{\otimes p}\otimes \omega^{p+k}))^{\otimes p-1-k}.\notag
\end{align}

In particular, for $p=2$ we have
\begin{align}
 (\det Rf_{*}L)^{\otimes 16}\cong&(\det Rf_{*}(L^{\otimes 2}))^{\otimes 7}\otimes(\det Rf_{*}(\omega\otimes(L^{\otimes 2})))^{\otimes~(-4)}\notag\\
                                     &\otimes\det Rf_{*}(\omega^{2}\otimes(L^{\otimes 2})).\notag
\end{align}
(II) The isomorphism in (I) is stable under base change, i.e., for any flat base extension $g :S^{'}\rightarrow S$ and the fiber product under base change:
$$\xymatrix{X^{'}\ar[r]^{g^{'}}\ar[d]_{f^{'}}&X\ar[d]^{f}\\
S^{'}\ar[r]^{g}&S }$$
Then there are canonical isomorphisms over $S^{'}$:
$$\xymatrix{g^{*}((\text{det}_{S}(Rf_{*}(L))^{\otimes p^{4}})\ar[rr]^{\cong}\ar[d]_{\cong}& &(\det_{S^{'}} Rf^{'}_{*}g^{'*}L)^{\otimes p^{4}}\ar[d]^{\cong}\\
  B\ar[rr]^{\cong}& &A}$$
\begin{align} 
 A=: &(\det Rf^{'}_{*}g^{'*}L^{\otimes p})^{\otimes~3p^{2}-3p+1}\otimes_{k=1}^{p-1}(\det Rf^{'}_{*}(g^{'*}L^{\otimes p}\otimes \omega^{'k}))^{\otimes k+1-3p}\notag\\
&\otimes_{k=0}^{p-2}(\det Rf^{'}_{*}(g^{'*}L^{\otimes p}\otimes \omega^{'p+k}))^{\otimes p-1-k}\notag
\end{align}
\begin{align} 
B=: &g^{*}((\det Rf_{*}L^{\otimes p})^{\otimes~3p^{2}-3p+1}\otimes_{k=1}^{p-1}(\det Rf_{*}(L^{\otimes p}\otimes \omega^{k}))^{\otimes k+1-3p}\notag\\
&\otimes_{k=0}^{p-2}(\det Rf_{*}(L^{\otimes p}\otimes \omega^{p+k}))^{\otimes p-1-k})\notag
\end{align}
where $\omega^{'}=\Omega_{f^{'}}$ is the relative differentials of the morphism $f^{'}$.
More precisely, the pull-back of the isomorphism in (I) for the morphism $f: X\rightarrow S$ coincides with the isomorphism in (I) for the morphism $f^{'}: X^{'}\rightarrow S^{'}$.

\end{theo}
According to Definition \ref{vdk} we made, there is an induced functor from the virtual category $V(X)$ to the Picard category $\mathscr{P}is_{S}$. The isomorphisms in (I) and (II) can be viewed as 
isomorphisms of line bundles because we didn't write out the degree of graded line bundles, but that is from the isomorphism of graded line bundles in the category $\mathscr{P}is_{S}$. 
Because for any two objects $(L,l)$ and $(M,m)$ in the category $\mathscr{P}is_{S}$, they are isomorphic if and only if $L\cong M$ and $l=m$. We will apply ideas appearing in  the proof of the $p$-th Adams-Riemann-Roch Theorem in the case of characteristic $p>0$ to our proof. To some extent, our theorem can be viewed as a variant of Deligne's functorial Riemann-Roch theorem.  Both of us  want to
give the expression for $(\det Rf_{*}L)^{\otimes k}$ by tensor product of $\det Rf_{*}L^{\otimes l}$, $\det Rf_{*}\omega^{\otimes m}$ and
$\det Rf_{*}\mathcal{O}$ with some power for some $k, l, m$.

\begin{proof}
Firstly, for any prime number $p$, we have
\begin{align}
&(\det Rf_{*}L)^{\otimes p^{4}}\notag\\
&\cong F^{*}_{S}(\det (Rf_{*}L))^{\otimes p^{3}}\notag\\
&\cong(\det \text{L}F^{*}_{S}(Rf_{*}L))^{\otimes p^{3}}\\
&\cong(\det Rf^{'}_{*}(J^{*}L))^{\otimes p^{3}}\\
&\cong\det Rf^{'}_{*}(p^{3}J^{*}L+(F_{*}\mathcal{O}_{X}-p)^{\otimes 3}\otimes J^{*}L)\\
&\cong\det Rf^{'}_{*}(p^{3}J^{*}L+((F_{*}\mathcal{O}_{X})^{\otimes 3}-3p(F_{*}\mathcal{O}_{X})^{\otimes 2}+3p^{2}(F_{*}\mathcal{O}_{X}))\otimes J^{*}L\notag\\
&~~- p^{3}J^{*}L)\\
&\cong\det Rf^{'}_{*}((F_{*}\mathcal{O}_{X})\otimes(p^{2}+p(p-F_{*}\mathcal{O}_{X})+(p-F_{*}\mathcal{O}_{X})^{2})
\otimes J^{*}L)\\
&\cong\det Rf_{*}(F^{*}(p^{2}+p(p-F_{*}\mathcal{O}_{X})+(p-F_{*}\mathcal{O}_{X})^{2})\otimes F^{*}_{X}L)\\
&\cong\det Rf_{*}((p^{2}+p(p-F^{*}F_{*}\mathcal{O}_{X})+(p-F^{*}F_{*}\mathcal{O}_{X})^{2})\otimes L^{\otimes p})\\
&\cong(\det Rf_{*}L^{\otimes p})^{\otimes 3p^{2}}\otimes(\det Rf_{*}(F^{*}F_{*}\mathcal{O}_{X}\otimes L^{\otimes p})))^{\otimes (-3p)}\notag\\
&~~\otimes\det Rf_{*}((F^{*}F_{*}\mathcal{O}_{X})^{\otimes 2}\otimes L^{\otimes p})\\
&\cong(\det Rf_{*}L^{\otimes p})^{\otimes~3p^{2}-3p+1}\otimes_{k=1}^{p-1}(\det Rf_{*}(L^{\otimes p}\otimes \omega^{k}))^{\otimes k+1-3p}\notag\\
&\otimes_{k=0}^{p-2}(\det Rf_{*}(L^{\otimes p}\otimes \omega^{p+k}))^{\otimes p-1-k}.
\end{align}

In fact, these isomorphisms are the consequences of properties of $\det$ and isomorphisms appearing in the $p$-th Adams-Riemann-Roch theorem in the case of characteristic $p>0$.
 Here, we explain them one by one. We will use the following diagram and some notations again.
$$\xymatrix{X\ar@/_/[ddr]_{f}\ar@{.>}[dr]|{F_{X/S}}\ar @/^/[rrd]^{F_{X}}&  & \\
                            & X^{'}\ar[r]_{J}\ar[d]^{f^{'}}&X\ar[d]^{f}\\
                             &S\ar[r]^{F_{S}}&S
                                 }$$
                                                                
We will continue to use $F$ to denote the relative Frobenius instead of $F_{X/S}$ for simplicity.
Firstly, by the definition of the extended determinant functor and the property of the absolute  Frobenius morphism $F^{*}_{S}$ (by Prop. \ref{frebad}),  $(\det (Rf_{*}L))^{\otimes p^{3}}$ is a 
line bundle and we have the isomorphism $(\det Rf_{*}L)^{\otimes p^{4}}\cong F^{*}_{S}(\det (Rf_{*}L))^{\otimes p^{3}}$.

Moreover, (1) follows from the fact that the pull-back commutes with the determinant functor by the property iii) of the extended determinant functor, where $\text{L}F^{*}_{S}$ is the left
derived functor of the functor $F^{*}_{S}$.

We get (2) because cohomology commutes with flat base change, i.e., \\$\text{L}F^{*}_{S}\cdot Rf_{*}\cong Rf^{'}_{*}\cdot\text{L}J^{*}$ (see \cite{Berth}, IV, Prop. 3.1.1).  Because $L$ is a line bundle, 
$\text{L}F^{*}_{S}$ is the same with $F^{*}_{S}$.

In (3), we introduce a new term $(F_{*}\mathcal{O}_{X}-p)^{\otimes 3}\otimes J^{*}L$. In Lemma \ref{l,f}, we know that $F_{*}\mathcal{O}_{X}$ is locally free of rank $p^{r}$ where $r$ is the relative dimension 
of the morphism $f$. Because $f$ is relatively smooth of dimension $1$ in our condition, $F_{*}\mathcal{O}_{X} -p$ is a virtual vector bundle of rank $0$.
According to Corollary \ref{trs}, 
 $\det Rf^{'}_{*}((F_{*}\mathcal{O}_{X}-p)^{\otimes 3}\otimes J^{*}L)$ is trivial.

After that, (4) is a expansion of (3) and (5) is a recombination of (4), by taking the term $F_{*}\mathcal{O}_{X}$ out and leaving $p-F_{*}\mathcal{O}_{X}$ out. 

(6) is direct from the projection formula (see \cite{Berth}, III, Pro. 3.7) and the fact $F^{*}_{X}=F^{*}J^{*}$. We know that $F_{X}^{*}$ has the same property with the $\psi^{p}$ in the case of characteristic $p>0$ by Prop. \ref{frebad}, i.e., $F_{X}^{*}(L)=L^{\otimes p}$.
By the functoriality of the functor $\det$, we  make combinations and go into (7).

In the $p$-th Adams-Riemann-Roch theorem in characteristic $p>0$,  we have the isomorphism $F^{*}F_{*}\mathcal{O}_{X}\cong \theta^{p}(\omega)$ (see Prop. \ref{gci} and some statements before Def. \ref{ad,1}). 
As in the Grothendieck group, there are equalities $F^{*}F_{*}\mathcal{O}_{X}=\tau(\omega)=1+\omega+\omega^{2}+,\cdots,+\omega^{p-1}$ in the virtual category $V(X)$.

Replacing $F^{*}F_{*}\mathcal{O}_{X}$ by the equality above in (7) and sorting out (7), we have
\begin{align}
&(\det Rf_{*}L)^{\otimes p^{4}}\notag\\
&\cong F^{*}_{S}(\det Rf_{*}L)^{\otimes p^{3}}\cong(\det \text{L}F^{*}_{S}(Rf_{*}L))^{\otimes p^{3}}\notag\\
&\cong(\det Rf_{*}L^{\otimes p})^{\otimes 3p^{2}}\otimes(\det Rf_{*}(F^{*}F_{*}\mathcal{O}_{X}\otimes L^{\otimes p}))^{\otimes -3p}\notag\\
&~~\otimes\det Rf_{*}((F^{*}F_{*}\mathcal{O}_{X})^{\otimes 2}\otimes L^{\otimes p})\notag\\
&\cong(\det Rf_{*}L^{\otimes p})^{\otimes~3p^{2}-3p+1}\otimes_{k=1}^{p-1}(\det Rf_{*}(L^{\otimes p}\otimes \omega^{k}))^{\otimes k+1-3p}\notag\\
&\otimes_{k=0}^{p-2}(\det Rf_{*}(L^{\otimes p}\otimes \omega^{p+k}))^{\otimes p-1-k}.\notag
\end{align}
These are isomorphisms (8) and (9), which finishes  the proof of isomorphisms in (I). In particular, for $p=2$ and by a direct computation we have 
\begin{align}
 (\det Rf_{*}L)^{\otimes 16}\cong&(\det Rf_{*}(L^{\otimes 2}))^{\otimes 7}\otimes(\det Rf_{*}(\omega\otimes L^{\otimes 2}))^{\otimes~(-4)}\notag\\
                                     &\otimes\det Rf_{*}(\omega^{2}\otimes(L^{\otimes 2})).\notag
\end{align}

For (II), there are well-known facts about base change, i.e., the smooth morphism is stable under base change (see \cite{Hart}, Chap. III, section 10) and the projective  morphism is
also stable under flat base change (see \cite{Liu}, 6.3.2). It means that $f^{'} :X^{'}\rightarrow S^{'}$ is projective and  smooth of relative dimension
$1$. Furthermore, for any line bundle $L$ on $X$, $Rf_{*}(L)$ is a strictly perfect complex in $\text{Parf}^{0}$ (see (3) of the section 2.1 and Rem. \ref{pcd}). Then we have $$g^{*}(\text{det}_{S}(Rf_{*}(L)))\cong\text{det}_{S^{'}}(\text{L}g^{*}(Rf_{*}(L)))\cong\text{det}_{S^{'}}(Rf^{'}_{*}(g^{'*}L))).$$
The first isomorphism is from iii) of the definition of the extended determinant functor. The second isomorphism is from the base-change formula (see \cite{Berth}, IV, Prop. 3.1.1), i.e., $\text{L}g^{*}Rf_{*}\cong Rf^{'}_{*}\text{L}g^{'*}$.
Because $L$ is a line bundle, $\text{L}g^{'*}L$ is same with $g^{'*}L$, which proves the horizontal isomorphism of the diagram in (II).

For the left vertical isomorphism in the diagram,  it is obvious, i.e., the pull-back for an isomorphism  is the pull-backs for two sides of the isomorphism, respectively. The pull-back of the right hand side of the isomorphism is just $B$.

For the isomorphism $B\cong A$, it results from the further expression of $B$. The pull-back of the right hand side of the isomorphism in (I) is the pull-back for $\det Rf_{*}L^{\otimes p}$, $\det Rf_{*}(L^{\otimes p}\otimes \omega^{k})$ and
$\det Rf_{*}(L^{\otimes p}\otimes \omega^{p+k})$, respectively. As proof in the horizontal isomorphism, for any vector bundle $F$ on $X$ we have
$$g^{*}(\text{det}_{S}(Rf_{*}(F))\cong\text{det}_{S^{'}}(Rf^{'}_{*}(g^{'*}F))).$$
Furthermore, these pull-backs are
$$g^{*}(\det Rf_{*}L^{\otimes p})\cong\text{det}_{S^{'}}(Rf^{'}_{*}(f^{'*}L)^{\otimes p});$$
\begin{align} 
g^{*}(\det Rf_{*}(L^{\otimes p}\otimes \omega^{k})&\cong\text{det}_{S^{'}}(Rf^{'}_{*}(g^{'*}(L^{\otimes p}\otimes\omega^{k}))\notag\\
&\cong\text{det}_{S^{'}}(Rf^{'}_{*}(g^{'*}L^{\otimes p}\otimes\omega^{'k}));\notag
 \end{align}
\begin{align}
g^{*}(\det Rf_{*}(L^{\otimes p}\otimes \omega^{p+k}))&\cong\text{det}_{S^{'}}(Rf^{'}_{*}(g^{'*}(L^{\otimes p}\otimes\omega^{p+k}))\notag\\
&\cong\text{det}_{S^{'}}(Rf^{'}_{*}(g^{'*}L^{\otimes p}\otimes\omega^{'p+k})).\notag
\end{align}
In the last two pull-backs, we use the fact that the differentials is stable under base change, i.e., $\omega^{'}\cong g^{'*}(\omega)$.
Putting these pull-backs together, this is just $A$.

Meanwhile, for the morphism $f^{'} :X^{'}\rightarrow S^{'}$ which satisfies the condition as the morphism $f$, 
we have the isomorphism for $g^{'*}L$, i.e.,

\begin{align} 
 (\det Rf^{'}_{*}&g^{'*}L)^{\otimes p^{4}}\notag\\
 &\cong(\det Rf^{'}_{*}g^{'*}L^{\otimes p})^{\otimes~3p^{2}-3p+1}\otimes_{k=1}^{p-1}(\det Rf^{'}_{*}(g^{'*}L^{\otimes p}\otimes \omega^{'k}))^{\otimes k+1-3p}\notag\\
&\otimes_{k=0}^{p-2}(\det Rf^{'}_{*}(g^{'*}L^{\otimes p}\otimes \omega^{'p+k}))^{\otimes p-1-k}.\notag
\end{align}
The right hand side is also $A$. This verifies the compatibility under base change.
\end{proof}

\begin{rem}
Our theorem is not a consequence of the Adams Riemann Roch theorem in $K$-theory. It results from the virtual category and the Picard category, which allows that our theorem is functorial. 
Compared with the general setting where the property of the Deligne pairing can be applicable and similar results can be obtained, our proof is not so complicated. In \cite{Den}, Eriksson defined the Adams operation and the Bott class on the virtual category and proved that the Adams Riemann Roch theorem was true in the localized Picard category.
Before that, he needs to define what the localized virtual category and the localized Picard category are. These definitions and proofs are impossible to state clearly in several pages.
In our theorem, the Adams operation and the Bott class defined on the virtual category is unnecessary. We emphasize more about the merits of the positive characteristic, which is one of our motivations.
\end{rem}

In \cite{Mum}, Mumford gave an isomorphism which is called Mumford's isomorphism now. We state it as follows:
 Let $f:C\rightarrow S$ be a flat 
local complete intersection generically smooth proper morphism with geometrically connected fibers of dimension $1$, with $S$ any connected normal Noetherian locally factorial scheme.
We denote $\lambda_{n}$ by $\det Rf_{*}\mathcal{\omega}^{\otimes n}$ for $\omega=\omega_{C/S}$ being the relative dualizing sheaf which is canonically isomorphic to the relative differentials of the morphism $f$. 
Then Mumford's isomorphism is $\det Rf_{*}(\omega^{n})\cong(\det Rf_{*}(\omega))^{6n^{2}-6n+1}$, whose original version is stronger than the present expression. 

By our theorem, we can get some results of a version of Mumford's isomorphism.
Under our condition, i.e., $f$ is a smooth and projective morphism of quasi-compact schemes with geometrically connected fibers of dimension $1$, then the isomorphism is also
$\det Rf_{*}(\omega^{n})\cong(\det Rf_{*}
(\omega))^{6n^{2}-6n+1}$, i.e., $\lambda_{n}=\lambda_{1}^{6n^{2}-6n+1}$ by notations.

\begin{coro}\label{mfi} 
 Suppose $p=2$ and let $f$ be as in theorem.  Then we have Deligne's Riemann-Roch theorem in \cite{Del} and Mumford's isomorphism.

\end{coro}
\begin{proof}
On the one hand, we have the isomorphism (I) in the previous theorem for any characteristic $p>0$. we will view the isomorphism as an isomorphism of graded
line bundles, even though we can't write up the corresponding degree of the graded line bundles. 

On the other hand, we have the Serre duality. For any vector bundle $F$ on $X$,  by the Grothendieck-Serre duality, $R\underline{Hom}(Rf_{*}F, \mathcal{O}) [-1]\cong Rf_{*}\underline{Hom}(F, \omega)$, where $\underline{Hom}$ is the hom functor for complexes of sheaves, one gets the isomorphism of graded line bundles 
between $\det Rf_{*}F$ and $\det Rf_{*}\underline{Hom}(F, \omega)$ (see \cite{Del}, Pag. 150). When we denote $\det Rf_{*}\mathcal{\omega}^{\otimes n}$ by $\lambda_{n}$, we also denote $\det Rf_{*}\mathcal{O}$ by  $\lambda_{0}$.
Firstly, let $F$ be the trivial bundle in the isomorphism of line bundles above. This is $\det Rf_{*}\mathcal{O} \cong \det Rf_{*}\mathcal{\omega}$, i.e., $\lambda_{0}\cong\lambda_{1}$.
Furthermore, let $F$ be a line bundle $L$ and  we have 
\begin{align}
\det Rf_{*}&((L-\mathcal{O})\otimes  (\omega\otimes L^{-1}- \mathcal{O}))\notag\\
&\cong \det Rf_{*}\omega \otimes  (\det Rf_{*}L)^{-1} \otimes (\det Rf_{*}(\omega \otimes L^{-1}))^{-1}\otimes 
 \det Rf_{*}\mathcal{O}\notag\\
 &\cong(\det Rf_{*}(L-\mathcal{O}))^{\otimes (-2)}.\notag
 \end{align} 
 
By the property of the Deligne pairing (Prop. \ref{d,p}),  then we have
\begin{align}
(\det Rf_{*}(L-\mathcal{O}))^{\otimes 2}&\cong\det Rf_{*}(-(L-\mathcal{O})\otimes  (\omega\otimes L^{-1}- \mathcal{O}))\notag\\
&\cong\det Rf_{*}((L-\mathcal{O})\otimes  ( \mathcal{O}-\omega\otimes L^{-1}))\notag\\
&\cong \langle L, \omega ^{-1}\otimes L\rangle. 
 \end{align}
By tensoring power $6$ of two sides of (10), we have $(\det Rf_{*}(L-\mathcal{O}))^{\otimes 12}\cong \langle L, \omega ^{-1}\otimes L\rangle^{\otimes 6}$.
Meanwhile, we consider the Deligne pairing  $\langle \omega, \omega \rangle$. According to the property of the Deligne pairing (Prop. \ref{d,p}), this is the following isomorphism
\begin{align}
\langle \omega, \omega\rangle& \cong \langle\det ( \omega-\mathcal{O}), \det(\omega-\mathcal{O})\rangle \notag\\
&\cong \det Rf_{*} (( \omega-\mathcal{O})\otimes (\omega-\mathcal{O}))\notag\\
&\cong \det Rf_{*}(\omega^{2})\otimes ( \det Rf_{*} \omega)^{\otimes (-2)}\otimes\det  Rf_{*} \mathcal{O}.\notag
 \end{align}
Taking $L$ for trivial bundle and $p=2$ in our theorem, we have
\begin{align}
(\det Rf_{*}\mathcal{O})^{\otimes 2^{4}}&\cong(\det Rf_{*}\mathcal{O})^{\otimes~3\cdot 2^{2}-3\cdot 2+1}\otimes(\det Rf_{*}\omega)^{\otimes~2-3\cdot 2}\notag\\
                                        &~~\otimes\det Rf_{*}(\omega^{2})\notag
\end{align}
i.e., $\lambda_{0}^{16}\cong \lambda_{0}^{7}\otimes \lambda_{1}^{-4}\otimes\lambda_{2}$. By $\lambda_{0}\cong\lambda_{1}$, that is $\lambda_{2}\cong \lambda_{1}^{13}$ and therefore,
hence $\langle \omega, \omega\rangle\cong\lambda_{2}\otimes \lambda_{1}^{(-2)}\otimes\lambda_{0}\cong\lambda_{0}^{12}$.

By the isomorphism $(\det Rf_{*}(L-\mathcal{O}))^{\otimes 12}\cong \langle L, \omega ^{-1}\otimes L\rangle^{\otimes 6}$, then we have the standard statement $$(\det Rf_{*}L)^{\otimes 12}\cong (\det Rf_{*}\mathcal{O})^{\otimes12}\otimes\langle L, \omega ^{-1}\otimes L\rangle^{\otimes 6}\cong\langle \omega, \omega\rangle\otimes\langle L, \omega ^{-1}\otimes L\rangle^{\otimes 6}.$$
which completely coincides with Deligne's statement in \cite{Del}.

If we use the property of the Deligne pairing  (Prop. \ref{d,p}) again as in Rem. \ref{explai}, the right hand side of the isomorphism in (10) is  
\begin{align} \langle L&,  ~\omega^{-1} \otimes L \rangle \notag\\
&\cong\det Rf_{*}(L^{2}\otimes \omega^{-1} ) \otimes (\det Rf_{*}L)^{-1}\otimes (\det Rf_{*}L\otimes\omega^{-1})^{-1}\otimes \det Rf_{*}(\mathcal{O}).\notag
\end{align} 
Then the isomorphism (10) becomes
 \begin{align}
(\det& Rf_{*}L)^{\otimes 2}\otimes (\det Rf_{*}\mathcal{O})^{\otimes (-2)}\\
&\cong\det Rf_{*}(L^{2}\otimes \omega^{-1} ) \otimes (\det Rf_{*}L)^{-1}\otimes (\det Rf_{*}L\otimes\omega^{-1})^{-1}\otimes \det Rf_{*}(\mathcal{O}).\notag
\end{align} 
Let $L$ be $n$-th power $\omega^{n}$ of the sheaf of differentials $\omega$ in (11). The isomorphism (11) is $\lambda_{n}^{2}\otimes \lambda_{0}^{\otimes(-2)}\cong\lambda_{2n-1}\otimes\lambda_{n}^{(-1)}\otimes \lambda_{n-1}^{\otimes (-1)}\otimes\lambda_{0}.$
In our proof of Deligne's Riemann Roch theorem, we already have the isomorphism $\lambda_{2}\cong \lambda_{1}^{13}$ which is just Mumford's isomorphism for $n=2$.
For general $n$ in Mumford's isomorphism , let $L$ be the bundle $\omega^{n}$ in our theorem. This is  
\begin{align}
(\det Rf_{*}\omega ^{n})^{\otimes 16}&\cong(\det Rf_{*}\omega ^{2n})^{\otimes~7}\otimes(\det Rf_{*}(\omega\otimes(\omega ^{2n})))^{\otimes(-4)}\notag\\
                                        &~~\otimes\det Rf_{*}(\omega^{2}\otimes(\omega ^{2n}))\notag
\end{align}
i.e., $\lambda_{2n+2}\cong\lambda_{n}^{16}\otimes\lambda_{2n}^{(-7)}\otimes\lambda_{2n+1}^{4}$. Plus the isomorphism $\lambda_{n}^{2}\otimes \lambda_{0}^{\otimes(-2)}\cong\lambda_{2n-1}\otimes\lambda_{n}^{(-1)}\otimes \lambda_{n-1}^{\otimes (-1)}\otimes\lambda_{0}$, 
by induction, that is just Mumford's isomorphism 
$\lambda_{n}\cong\lambda_{1}^{\otimes (6n^{2}-6n+1)}$ for $p=2$. 
\end{proof}

\begin{rem} 
The original proof of Mumford's isomorphism (see \cite{Mum}, Pages 99-110) is a  calculation by using Grothendieck-Riemann-Roch
and the facts which is referred to the Picard-group of the moduli-functor of stable curves and so on. In our corollary, it is a special case of our theorem
by taking the line bundle to be the trivial bundle and the Serre duality. For any prime number $p>2$, we have an analogous expression to Mumford's isomorphism. This is explained as follows: 
Let $L$ be the trivial bundle in our theorem. This is the isomorphism
 \begin{align}
 \lambda_{0}^{p^{4}} \cong &  \lambda_{0}^{3p^{2}-3p+1} \otimes\lambda_{1}^{2-3p} \otimes\lambda_{2}^{3-3p}\otimes \ldots\otimes \lambda_{p-1}^{p-3p}\otimes\notag\\
      & \lambda_{p}^{p-1}\otimes  \lambda_{p+1}^{p-2}\otimes\ldots \otimes\lambda_{2p-2}.\notag                         
      \end{align}
Given the isomorphism $\lambda_{0}\cong\lambda_{1}$,  then we have 
\begin{align}
 \lambda_{2p-2}\cong &\lambda_{1}^{p^{4}-3p^{2}+6p-3}\otimes\lambda_{2}^{3p-3}\otimes \ldots\otimes \lambda_{p-1}^{3p-p}\otimes\notag\\   
   & \lambda_{p}^{1-p}\otimes  \lambda_{p+1}^{2-p}\otimes\ldots \otimes\lambda_{2p-3}^{\otimes(-2)}.\notag      
\end{align}
More generally, let $L$ be $\omega^{n}$ in our theorem. Then we have 
 \begin{align}
 \lambda_{n}^{p^{4}} \cong &  \lambda_{np}^{3p^{2}-3p+1} \otimes\lambda_{np+1}^{2-3p} \otimes\lambda_{np+2}^{3-3p}\otimes \ldots\otimes \lambda_{np+p-1}^{p-3p}\otimes\notag\\
      & \lambda_{np+p}^{p-1}\otimes  \lambda_{np+p+1}^{p-2}\otimes\ldots \otimes\lambda_{np+2p-2}.\notag                         
      \end{align}
\end{rem}

\section*{Acknowledgements}
This is a note from my thesis in order to obtain the doctoral degree. I want to express my sincere thanks for my advisor Professor Damian R\"{o}ssler. The note results from his advice and encouragement, including
some mistakes corrected after some discussions. Also, my appreciates is given to my referees, Professor Bernhard K\"{o}ck and Vincent Maillot, who
read my thesis carefully and provide some helpful opinions. Finally, I will thank every author in my references whose academic resource is greatly help
for the note.

\end{document}